\documentclass[10pt]{article}
\usepackage{hyperref}
\usepackage{blkarray}

\usepackage{amsmath}
\usepackage{amssymb}
\usepackage{amsfonts}
\usepackage{amsthm}
\usepackage{enumerate}
\usepackage{multicol}
\usepackage{colortbl}

\usepackage{pgf}
\usepackage{tikz}
\usepackage{xcolor}
\usetikzlibrary{arrows}
\usetikzlibrary{patterns}
\usetikzlibrary{petri}

\usepackage[english]{babel}
\usepackage{tabularx}
\usepackage{subfig}

\newcommand{\rank}{{\rm rank}}
\newcommand{\p}{p}

\newcommand{\bm}{m}

\newcommand{\C}{\mathcal C}

\newcommand{\pog}{\langle G, \bm \rangle} 
\newcommand{\pogp}{\langle G', \bm' \rangle} 
\newcommand{\pogT}{\langle G, \bm_T \rangle}

\newcommand{\Tor}{\mathcal T_0} 
\newcommand{\Torx}{\mathcal T_x} 
\newcommand{\T}{\mathcal T} 

\newcommand{\pofw}{(\langle G, \bm \rangle,\p)} 
\newcommand{\pofwT}{(\langle G, \bm_T \rangle,\p)} 

\newcommand{\R}{\mathbf R} 

\definecolor{desk}{rgb}{.345, .306, .216}
\definecolor{vancouver}{rgb}{.412, .412,.412}
\definecolor{beetle}{rgb}{.180, .161, .102}
\definecolor{bluey}{rgb}{.235, .380, .415}
\definecolor{melon}{rgb}{1, .259, .259}
\definecolor{vneck}{rgb}{.596, .282, .376}
\definecolor{pink}{rgb}{.918, .122, .545}
\definecolor{mango}{rgb}{1, .8, .267}
\definecolor{lips}{rgb}{.541, .074, .239}
\definecolor{sage}{rgb}{.522, .604, .247}
\definecolor{moss}{rgb}{.184, .224, .129}
\definecolor{cumin}{rgb}{.6, .580, 0}
\definecolor{lichen}{rgb}{.745, .998, .729}
\definecolor{rain}{rgb}{.780, .812, .706 }
\definecolor{cloud}{rgb}{.961, .976, .870}
\definecolor{couch}{rgb}{.8, 1, .2}
\definecolor{cement}{rgb}{.678, .682, .549}
\definecolor{sky}{rgb}{.278, .514, 1}

\newtheorem{thm}{Theorem}[section]

\newtheorem{lem}[thm]{Lemma}
\newtheorem{prop}[thm]{Proposition}

\theoremstyle{remark}

\theoremstyle{remark}

\theoremstyle{remark}

\begin{document}

\title{Periodic Rigidity on a Variable Torus Using Inductive Constructions}
\author{
{A. Nixon \thanks{tony.nixon@bristol.ac.uk, Heilbronn Institute for Mathematical Research, School of Mathematics, University of Bristol, U.K.} ~and E.  Ross \thanks{elissa@mathstat.yorku.ca, Department of Mathematics and Statistics, York University, Canada}
}
}
\date{}
\maketitle

\begin{abstract} 
In this paper we prove a recursive characterisation of generic rigidity for frameworks periodic with respect to a partially variable lattice. We follow the approach of modelling periodic frameworks as frameworks on a torus and use the language of gain graphs for the finite counterpart of a periodic graph. In this setting we employ variants of the Henneberg operations used frequently in rigidity theory. 
\end{abstract}


\section{Introduction}

Given an embedding of a graph into Euclidean space as a collection of stiff bars and flexible joints (a framework), when is it possible to continuously deform the structure into a non-congruent position without breaking connectivity or changing the bar lengths (is the framework rigid or flexible)? This is the fundamental question in rigidity theory. The subject has many obvious applications, for example, in molecular biology, structural engineering and computer aided design \cite{CountingFrameworks}.

Typically the question is $NP$-hard \cite{Abbott} but for generic embeddings in the plane there is a complete combinatorial description.

\begin{thm}[Henneberg \cite{Henneberg}, Laman \cite{Laman}, Maxwell \cite{Maxwell}]\label{thm:Laman}
Let $G=(V,E)$ and let $p$ be a generic embedding into $\mathbb{R}^2$. Then the following are equivalent:
\begin{enumerate}
\item[(1)] the framework $(G,p)$ is generically minimally rigid,
\item[(2)] $G$ satisfies $|E|=2|V|-3$ and $|E'|\leq 2|V'|-3$ for every $G'=(V',E')\subset G$,
\item[(3)] $G$ can be constructed from a single edge by recursively adding vertices of degree $2$ and by removing an edge and adding vertices of degree $3$ adjacent to the ends of the old edge.
\end{enumerate}
\end{thm}
We will assume some basic familiarity with the rigidity of finite frameworks, those unfamiliar may wish to consult \cite{GraverServatius2}, \cite{CountingFrameworks} or \cite{MatroidsRigidStructures}. The operations defined in {\it (3)} will be referred to as Henneberg operations, \cite{Henneberg}.

Recently a lot of attention has been paid to finding analogues of Theorem \ref{thm:Laman} for infinite graphs, specifically for periodic graphs (those with a finite quotient in an appropriate sense). Particularly \cite{ThesisPaper1} and \cite{ThesisPaper2} developed the approach of considering periodic frameworks as frameworks embedded on a torus. Indeed in \cite{ThesisPaper2} a characterisation was given for the case where the lattice (or torus) is fixed.

The purpose of this paper is to prove the following result concerning the case when one of the lattice vectors is allowed to change.

\begin{thm}\label{thm:rigidityinduction}
A labelled graph is generically minimally rigid on the partially variable torus if and only if it can be derived from a single loop by gain-preserving Henneberg operations.
\end{thm}

The gain-preserving Henneberg operations are extensions of the recursive moves in Theorem \ref{thm:Laman} (3) to include labelled edges,  and will be formally defined in Section \ref{sec:h1h2}. One of the intricacies of the theorem is that we are working with a class of graphs, the $P(2,1)$-graphs (defined in Section \ref{sec:necessary conds}), that fall in between $(2,1)$-tight graphs and $(2,2)$-circuits. It is, in part, for this reason that we do not adopt a matroid theoretic approach in this paper. 


Throughout the paper, we focus on a particular type of variable torus, namely one which is variable in the $x$-direction only. In fact, the results are much more general, and apply to frameworks on a torus which is variable in the $y$-direction only, and frameworks on a torus which has a variable angle between two fixed-length generators. We may also apply this result for a full characterisation of the generic rigidity of frieze-type patterns, that is, frameworks which are periodic in one direction only (interpreted as frameworks on a variable cylinder). These variations are discussed in Section \ref{sec:extensions}. 

The study of periodic frameworks is a topic which has experienced a surge of interest over the past decade \cite{BorceaStreinuII,periodicFrameworksAndFlexibility,MalesteinTheran,myThesis,DeterminancyRepetitive}. This work has been motivated in part by questions arising in materials science about the structural properties of {\it zeolites}, a type of mineral with a repetitive, micro-porous structure \cite{flexibilityWindow}. Furthermore, there may be physical meaning associated with certain restrictions of the fully variable torus. It has been suggested that the time scales of atomic movement are significantly different from those of lattice deformation \cite{ThorpePrivate}. 

In this paper we take an inductive approach to the problem of characterizing the generic rigidity of periodic frameworks. That is, we define a collection of local graph-theoretic moves which characterize the class of generically rigid periodic frameworks on a partially variable torus. The inductive method has the advantage of being easy to state and understand. Furthermore, while finding an inductive construction for a particular graph does not in general make for fast algorithms,  once we have such an inductive sequence, it offers an immediate certificate of the rigidity of that framework. 

\subsection{Results in Context}

The basic theory of periodic frameworks has been well formalized by Borcea and Streinu \cite{periodicFrameworksAndFlexibility}. The approach we use here is based on the presentation appearing in \cite{ThesisPaper1}. In that paper, the links between the approach of \cite{periodicFrameworksAndFlexibility} and the present methodology are outlined in detail. 

In \cite{ThesisPaper2}, Ross proved an inductive characterisation of the generic rigidity of two-dimensional periodic frameworks on a {\it fixed torus}, that is, a torus with no variability. The methods used here build on those results.  In \cite{MalesteinTheran} Malestein and Theran proved a characterisation of generic minimal rigidity of two-dimensional frameworks on the fully variable torus (three degrees of freedom). They obtain the result of Ross as a restriction of their more general theorem. However, their methods differ significantly from ours, in that they do not use an inductive characterisation. Indeed giving an inductive construction for the relevant class of graphs on the fully variable torus is an intriguing open problem.

\subsection{Outline of Paper}

In Section \ref{sec:Background} we recall the basic theory of periodic frameworks as frameworks on a torus. The following two sections state the relevant rigidity results for the partially variable torus, Maxwell-type necessary conditions and Henneberg constructions preserving rigidity of frameworks. In Section \ref{sec:p21} we prove some preliminary graph theory results, including an inductive construction of $P(2,1)$-graphs that may be of independent interest. The main body of the paper is contained in Section \ref{gainsection} where we present a case by case analysis showing that the appropriate gains are preserved by the construction operations. This allows us to prove our main theoretical result, Theorem \ref{thm:mainresult}, and hence to complete the proof of Theorem \ref{thm:rigidityinduction}. We describe some extensions of the work in Section \ref{sec:extensions}. The final section concludes the paper with some discussion of further work.

\section{Background}
\label{sec:Background}

A {\it periodic framework} in the plane is a locally finite infinite graph which is symmetric with respect to the free action of $\mathbb Z^2$. Such a framework has a finite number of vertex and edge orbits under the action of $\mathbb Z^2$. Full definitions and details can be found in the work of Borcea and Streinu, \cite{BorceaStreinuII,periodicFrameworksAndFlexibility}. The approach taken here is to consider periodic frameworks as orbit frameworks on a torus, as in \cite{ThesisPaper1,ThesisPaper2}. 

\subsection{Periodic Orbit Frameworks on the Variable Torus $\Torx^2$}

Let $\Torx^2 = \mathbb R^2 / {L_x\mathbb Z^2}$, where 
\[L_x = L_x(t) =  \left(\begin{array}{cc}x(t) & 0 \\y_1 & y_2\end{array}\right), \ y_1, y_2 \in \mathbb R.\] 
We call $\Torx^2$ the {\it $x$-variable torus}, and the matrix $L_x = L_x(t)$ is the lattice matrix. Similarly, let  
\[L_0 =  \left(\begin{array}{cc}x & 0 \\y_1 & y_2\end{array}\right), \ x, y_1, y_2 \in \mathbb R\] 
be the {\it fixed lattice matrix}, and we call the quotient space $\Tor^2 = \mathbb R^2 / {L_0\mathbb Z^2}$ the {\it fixed torus}. 

For a graph $G$, we will use the notation $V(G)$ and $E(G)$ to refer to the vertex and edge sets of $G$, if not explicitly named.  A {\it periodic orbit framework} $\pofw$ on $\Torx^2$ consists of a labelled, directed multigraph $\pog$ together with a position $p$ of the vertices $V(G)$ on the variable torus $\Torx^2$. $\pog$ is called a {\it gain graph} \cite{Zaslavsky,TopologicalGraphTheory}, and it encodes the way in which the graph $G$ is ``wrapped" around the torus. Specifically, $\pog$ is composed of a directed multigraph $G=(V, E)$, and a labelling $m: E^+ \rightarrow \mathbb Z^2$. The labelling of the edges is invertible, meaning that we may write
\[e=\{v_i, v_j; m_e\} = \{v_j, v_i; -m_e\}.\]
From the periodic orbit graph we may define the derived periodic graph $G^m$, which has vertex set $V^m = V \times \mathbb Z^2$, and edge set $E^m = E \times \mathbb Z^2$. If $e = \{v_i, v_j; m_e\} \in E\pog$, then the edge $(e, z) \in E^m, z \in \mathbb Z^2$ connects the vertices $(v_i, z)$ and $(v_j, z+m_e) \in V^m$. In this way, the periodic orbit graph $\pog$ is a kind of ``recipe" for the infinite derived graph $G^m$ (see Figure \ref{fig:perOrbitGraph}). Furthermore, the automorphism group of $G^m$ contains $\mathbb Z^2$. 

\begin{figure}[htbp]
\begin{center}
\subfloat[$\pog$]{\includegraphics[width=1.5in]{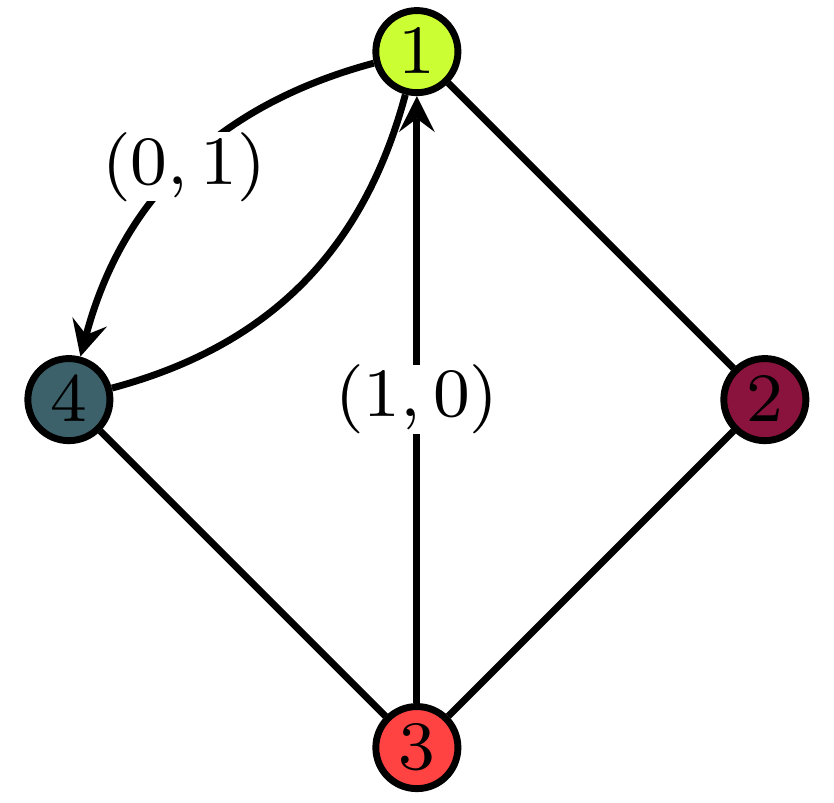}}\hspace{.5in}
\subfloat[$G^m$]{\includegraphics[width=1.5in]{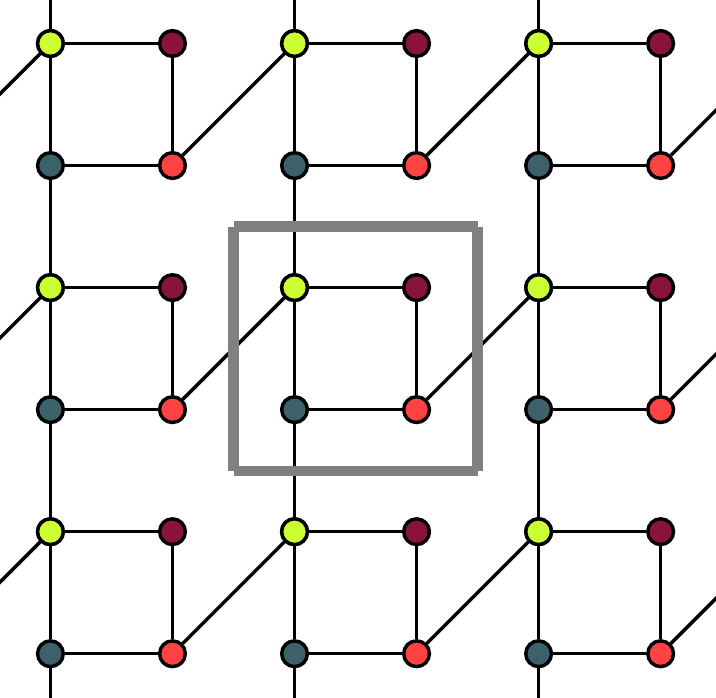}}
\caption{The periodic orbit graph $\pog$ and its corresponding derived graph $G^m$.}
\label{fig:perOrbitGraph}
\end{center}
\end{figure}

In a similar fashion we use the periodic orbit framework $\pofw$ to define an infinite periodic framework, the {\it derived periodic framework} $(\langle G^m, L_x \rangle, p^m)$, where $p^m: V^m \rightarrow \mathbb Z^2$ is given by 
\[p^m(v, z) = p(v) + zL_x, \ \textrm{where}\ v \in V, z \in \mathbb Z^2.\]

An {\it infinitesimal motion} of $\pofw$ on $\Torx^2$ is an element $(u, u_x) \in \mathbb R^{2|V| +1}$, where
\[u: V \rightarrow \mathbb R^2, \ \textrm{and} \ u_x:x(t) \rightarrow \mathbb R,\]
such that
\begin{equation}
(p_i - p_j - m_e L_x) \cdot (u_i - u_j - (u_x(x), 0)) = 0 \ \textrm{for all} \ \{v_i, v_j; m_e\} \in E\pog.
\label{eq:infMot}
\end{equation}
If $u_x = 0$, and $u: V \rightarrow z, z \in \mathbb R^d$ (i.e. $u$ is a translation), then we say that $(u, u_x)$ is a trivial motion of $\pofw$ on $\Torx^2$. If the only infinitesimal motions of a framework $\pofw$ on $\Torx^2$ are trivial, then we say that $\pofw$ is {\it infinitesimally rigid} on $\Torx^2$. 

Any infinitesimal motion $(u, u_x)$ of $\pofw$ on the $x$-variable torus $\Torx^2$ for which $u_x = 0$ is also an infinitesimal motion of $\pofw$ on the fixed torus $\Tor^2$. The definitions of trivial motions and infinitesimal rigidity on $\Tor^2$ are the same as for $\Torx^2$. 

\subsection{The Rigidity Matrix $\R_x$}

The {\it rigidity matrix} $\R_x$ permits us to simultaneously solve the equations (\ref{eq:infMot}) for the space of infinitesimal motions of $\pofw$. It is an $|E| \times (2|V| + 1)$ matrix with one row corresponding to each edge, two columns corresponding to each vertex, and a single column corresponding to the variable lattice element $x(t)$. The row corresponding to the edge $\{v_i, v_j; m_e\}$ is as follows:

\[  \renewcommand{\arraystretch}{0.8}
     \bordermatrix{  &&  i &  &  j & &  x(t)    \cr
 &    0 \cdots 0 & p_i - (p_j+m_eL_x) & 0  \cdots  0 & (p_j+m_eL_x) - p_i & 0  \cdots  0  &   (m_e)_x[p_i - (p_j+m_eL_x)]_x\cr 
     }, \]
where the entries under $i$ and $j$ are actually $2$-tuples. By $(m_e)_X$ we mean the $x$-component of $m_e \in \mathbb Z^2$. The kernel of this matrix is the space of infinitesimal motions of $\pofw$ on $\Torx^2$, and we may write 
\[\R_x\pofw \cdot (u, u_x)^T = 0,\]
where $(u, u_x) \in \mathbb R^{2|V|+1}$ is as described above. 

A framework on $\Torx^2$ always has a two-dimensional space of trivial infinitesimal motions, generated by the unit translations. It follows that the kernel of the rigidity matrix always has dimension at least $2$. Furthermore, since a framework is infinitesimally rigid on $\Torx^2$ if and only if the only infinitesimal motions are trivial (i.e. are translations), we have the following result:
\begin{thm}\label{thm:rigidiffrank}
A periodic orbit framework $\pofw$ is infinitesimally rigid on the $x$-variable torus $\Torx^2$ if and only if the rigidity matrix $\R_x\pofw$ has rank $2|V| - 1$.
\end{thm}

\subsection{The $T$-gain Procedure on $\Torx^2$}
\label{sec:TGainsPreserveInfinitesimalRigidity}

In \cite{ThesisPaper1} Ross described the {\it $T$-gain procedure}, and showed that the rigidity matrices corresponding to two $T$-gain equivalent periodic orbit frameworks have the same rank. We now extend this to the variable torus case. See Figure \ref{fig:Tvoltage} for an example.

The {\it net gain} on a cycle in a periodic orbit graph $\pog$ is the sum of the gains on the edges of an oriented cycle of $G$, where the gains are appropriately multiplied by $\pm 1$, depending on the direction of traversal. The T-gain procedure is a procedure that can be used to easily identify the net gains on the cycles of a periodic orbit graph $\pog$. As we will soon see, the rigidity of frameworks on $\Torx^2$ is generically characterized by the net gains on the cycles of the periodic orbit graph (Theorem \ref{thm:rigidityinduction}). The $T$-gain procedure will thus be an essential proof technique which we use to show the necessity of the conditions in our main result (Proposition \ref{prop:necessary}). The $T$-gain procedure is defined in \cite{TopologicalGraphTheory} for general gain (voltage) graphs. 

\begin{verse}
\begin{figure}[h!]
\begin{center}
\begin{tikzpicture}[->,>=stealth,shorten >=1pt,auto,node distance=2.8cm,thick, font=\footnotesize] 
\tikzstyle{vertex1}=[circle, draw, fill=couch, inner sep=.5pt, minimum width=3.5pt, font=\footnotesize]; 
\tikzstyle{vertex2}=[circle, draw, fill=melon, inner sep=.5pt, minimum width=3.5pt, font=\footnotesize]; 
\tikzstyle{voltage} = [fill=white, inner sep = 0pt,  font=\scriptsize, anchor=center];

	\node[vertex1] (1) at (-1.3,0)  {$1$};
	\node[vertex1] (2) at (1.3,0) {$2$};
	\node[vertex1] (3) at (0, 2) {$3$};
		\draw[thick] (1) -- node[voltage] {$(1,2)$} (2);
		\draw[thick] (2) -- node[voltage] {$(0,1)$} (3);
	\draw[thick] (3) edge  [bend right]  node[voltage] {$(3,1)$} (1);
	\draw[thick] (3) edge  [bend left] node[voltage] {$(1,-1)$} (1);
	
	\node[font=\normalsize] at (0, -1) {(a)};
	
	\pgftransformxshift{4cm};
	
		\node[vertex1] (1) at (-1.3,0)  {$1$};
	\node[vertex1] (2) at (1.3,0) {$2$};
	\node[vertex1] (3) at (0, 2) {$3$};
		\draw[very thick, red] (1) -- node[voltage] {$(1,2)$} (2);
		\draw[thick] (2) -- node[voltage] {$(0,1)$} (3);
	\draw[thick] (3) edge  [bend right]  node[voltage] {$(3,1)$} (1);
	\draw[very thick, red] (3) edge  [bend left] node[voltage] {$(1,-1)$} (1);
	
	\node[blue] at (0, 2.3) {$u$};
	\node[blue] at (-1.6, -.4) {$(1, -1)$};
	\node[blue] at (1.6, -.4) {$(2, 1)$};
	
	\node[font=\normalsize] at (0, -1) {(b)};
	
	\pgftransformxshift{4cm};
	
	\node[vertex1] (1) at (-1.3,0)  {$1$};
	\node[vertex1] (2) at (1.3,0) {$2$};
	\node[vertex1] (3) at (0, 2) {$3$};
		\draw[very thick, red] (1) -- node[voltage] {$(0,0)$} (2);
		\draw[thick] (2) -- node[voltage] {$(2,2)$} (3);
	\draw[thick] (3) edge  [bend right]  node[voltage] {$(4,0)$} (1);
	\draw[very thick, red] (3) edge  [bend left] node[voltage] {$(0,0)$} (1);
	
	\node[font=\normalsize] at (0, -1) {(c)};

\end{tikzpicture}
\caption{A gain graph $\pog$ in (a), with identified tree $T$ (in red), root $u$, and $T$-potentials in (b). The resulting $T$-gain graph $\langle G, \bm_T \rangle$ is shown in (c).  \label{fig:Tvoltage}}
\end{center}
\end{figure}
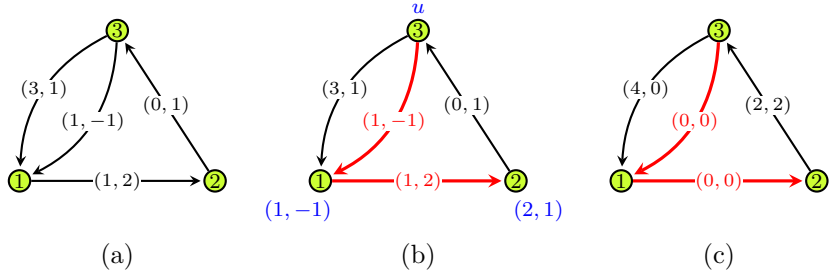\end{verse}

\noindent {\bf $T$-gain Procedure}
\begin{enumerate}
	\item Select an arbitrary spanning tree $T$ of $G$, and choose a vertex $u$ to be the root vertex.
	\item For every vertex $v$ in $G$, there is a unique path in the tree $T$ from the root $u$ to $v$. Denote the net gain along that path by $\bm(v, T)$, and we call this the {\it $T$-potential} of $v$. \index{T-gain procedure!T-potential} Compute the $T$-potential of every vertex $v$ of $G$. 
	\item Let $e$ be a plus-directed edge of $G$ with initial vertex $v$ and terminal vertex $w$. Define the {\it $T$-gain} of $e$, $\bm_T(e)$ to be 
	$$\bm_T(e) = \bm(v, T) +  \bm(e) - \bm(w, T).$$ 
Compute the $T$-gain of every edge in $G$. Note that the $T$-gain of every edge of the spanning tree will be zero. 
\end{enumerate}

\begin{thm}
Let $\pofw$ be a periodic orbit framework on $\Torx^2$. Then $\rank \R_x\pofw = \rank \R_x(\pogT, p')$,  where $\p':V \rightarrow \mathbb R^2$ is given by $\p'_i = \p_i + \bm_T(v_i)$.
\label{thm:TgainsPreserveRigidityFlex}
\end{thm}

\begin{proof}
Suppose that a set of rows is dependent in $\R_x\pofw$. Then there exists a vector of scalars, say $\omega = [\begin{array}{ccc} \omega_1 & \cdots & \omega_{|E|} \end{array}]$ such that 
\[\omega \cdot \R_x\pofw = 0.\]
For each vertex $v_i \in V$ the column sum of $\R_x(\pofw)$ becomes 
\begin{equation}
\sum_{ e_{\alpha} \in E_+} \omega_{e_{\alpha}} (p_i - (p_j+m_{e_{\alpha}})) + 
\sum_{ e_{\beta} \in E_-} \omega_{e_{\beta}} (p_i - (p_k-m_{e_{\beta}})) = 0,
\label{eqn:tGainFlex}\end{equation}
where $E_+$ and $E_-$ are the edges directed out from and into vertex $i$ repectively. 
In \cite{ThesisPaper1}, it was demonstrated that (\ref{eqn:tGainFlex}) is equivalent to the following:
\begin{eqnarray}
& \sum_{ e_{\alpha} \in E_+} \omega_e \Big(\p_i+ \bm_T(v_i)  - (\p_j + \bm_T(v_j)) -\bm_T(e)\Big) + \nonumber \\
& \hspace{1in} \sum_{ e_{\beta} \in E_-} \omega_e \Big(\p_i+ \bm_T(v_i)  - (\p_j + \bm_T(v_j)) +\bm_T(e)\Big) = 0
\label{eqn:tGainFlex2}
\end{eqnarray}
which is the column sum of the column of  $\R_x(\pogT, p')$ corresponding to the vertex $v_i$.

Since we are working with the variable torus, we have one additional column corresponding to the flexibility of the $x$-direction.  %
%
%
We will show that if there exists a vector of scalars $\omega = [\begin{array}{ccc} \omega_1 & \cdots & \omega_{|E|} \end{array}]$ such that 
\[\omega \cdot \R_x\pofw = 0,\]
then $\omega \cdot \R_x\pofwT = 0$ too. Since the first $2|V|$ columns are exactly as in the fixed torus case, we need only show this holds for the new column. 

Consider the column sum corresponding to the columns of the lattice elements in $\R_x(\pogT, p')$:
\begin{equation}
\sum_{e \in E} \omega_e  \Big(m_T(e)\Big[(\p_i+ \bm_T(v_i))  - (\p_j + \bm_T(v_j)) -\bm_T(e)\Big]\Big)_x.
\label{eqn:flexTorusTGain}
\end{equation}
Recall that $\bm_T(e) = \bm_T(v_i) + \bm(e) - \bm_T(v_j)$, where $m_T(v_i)$ represents the $T$-potential of the vertex $v_i$ (the $T$-potential of a vertex $v_i$ is the net gain on the directed path along $T$ from the root vertex). 
Expanding (\ref{eqn:flexTorusTGain}), we obtain
\begin{equation}
\sum_{e \in E} \omega_e \Big( m(e)\big[\cdots \big]  + m_T(v_i)\big[\cdots \big] - m_T(v_j)\big[\cdots \big]\Big)_x,
\label{eqn:flexTorusTGain2}
\end{equation}
where $ \big[\cdots \big] = \big[(\p_i+ \bm_T(v_i))  - (\p_j + \bm_T(v_j)) -\bm_T(e)\big]$. 
We know that 
\begin{align*}
&\sum_{e \in E} \omega_e \Big( m(e)\Big[(\p_i+ \bm_T(v_i))  - (\p_j + \bm_T(v_j)) -\bm_T(e)\Big]\Big)_x\\
= & \sum_{e \in E} \omega_e \Big( m(e)\Big[\p_i  - \p_j -\bm(e)\Big]\Big)_x \\
= & 0, \ \ \ \textrm{since} \ \omega \cdot \R_x\pofw = 0.
\end{align*}
Now note that $\bm_T(v_i)$ and $\bm_T(v_j)$ have one of $|V|$ different values. Grouping (\ref{eqn:flexTorusTGain2}) according to these values, we obtain
\[\sum_{i=1}^{|V|}  (m_T(v_i))_x\Big[\sum_{j: (i,j) \in E} \omega_e(\p_i+ \bm_T(v_i)  - (\p_j + \bm_T(v_j)) -\bm_T(e)) \Big]_x,\]
where each edge is counted exactly twice, once for its initial vertex and once for its terminal vertex, with sign depending on the orientation of the edge. But by (\ref{eqn:tGainFlex2}), the sum inside the square brackets is zero, since it represents the column sum at any vertex. Hence (\ref{eqn:flexTorusTGain}) is also zero. The same argument also works in reverse, which proves the claim. %
\end{proof}

\subsection{Periodic Orbit Frameworks on the Fixed Torus $\Tor^2$}
From the rigidity matrix for the $x$-variable torus, we can obtain the rigidity matrix for frameworks on the fixed torus, simply by striking out the column corresponding to $x(t)$. We are left with an $|E| \times 2|V|$ matrix $\R_0$, and a periodic orbit framework $\pofw$ is infinitesimally rigid on $\Tor^2$ if and only if the rank of $\R_0\pofw$ is $2|V| - 2$ \cite{ThesisPaper2}. Frameworks on the fixed torus are the subject of the papers \cite{ThesisPaper2,ThesisPaper1}, and we record only the main result.

For brevity throughout we use the following terminology.
Let $G=(V, E)$ be a graph. We say that $G$ is {\it $(k, \ell)$-sparse} if all subgraphs $G' = (V', E')$ of $G$ satisfy $|E'| = k|V'| - \ell$. If in addition, $G$ satisfies $|E| = k|V| - \ell$, we say that $G$ is {\it $(k, \ell)$-tight}. 

Let $\pog$ be a periodic orbit graph, where $G$ is $(2,2)$-tight. We say that the gain assignment $m: E^+ \rightarrow \mathbb R^2$ is {\it constructive} if every subgraph $G' \subset G$ with exactly $|E'| = 2|V'| - 2$ edges contains some cycle with non-trivial net gain. For example, the periodic orbit graph $\pog$ pictured in Figure \ref{fig:Tvoltage}(a) has a constructive gain assignment. Note further that the $T$-gain procedure preserves the net gains on cycles, and therefore the graph $\pogT$ pictured in (c) also has a constructive gain assignment. 

\begin{thm}\label{thm:fixedtorus}
The periodic orbit graph $\pog$ is minimally rigid on the fixed torus $\Tor^2$ if and only if $G$ is $(2,2)$-tight, and $m$ is a constructive gain assignment. 
\end{thm}

\subsection{1-Dimensional Frameworks}
\label{sec:1dim}

The basic ideas of the rigidity of finite graphs on the line can be found in \cite{MatroidsRigidStructures} or \cite{CountingFrameworks}. The key result is that a graph $G$  is rigid as a $1$-dimensional framework if and only if it is connected. 
$1$-dimensional periodic frameworks were discussed in \cite{myThesis} and \cite{MalesteinTheran}.  We give a brief summary.

Just as we map $2$-periodic frameworks onto the torus, we may view $1$-periodic frameworks as graphs on a circle. Such graphs may be on a circle of fixed circumference $x$ ({\it the fixed circle}), or they may be on a circle that is allowed to change circumference $x(t)$ ({\it the flexible circle}). We denote the fixed circle by $\Tor^1$, and the flexible  circle by $\T^1$. 

In either case, a $1$-periodic orbit framework is the pair $\pofw$, with $m: E \rightarrow \mathbb Z$, and $\p:V \rightarrow [0, x)$, where $x$ is either a fixed element of $\mathbb R$ for the fixed circle, or $x = x(t)$ is a continuous function of time for the flexible circle. We assume further that $\p$ maps the endpoints of any edge to distinct locations in $[0, x)$, thereby avoiding edges of length zero. 

Consider a periodic orbit framework $\pofw$ on the fixed circle. 
For consistency with the notation used in the remainder of this paper, let $L_0$ be the $1 \times 1$ matrix $[x]$, where the circumference of the fixed circle is $x$.  The $|E| \times |V|$ rigidity matrix in this case will have one row corresponding to each edge $\{i, j; m\}$:
\[  \renewcommand{\arraystretch}{0.8}
     \bordermatrix{  &&  i &  &  j &      \cr
 &    0 \cdots 0 & p_i - (p_j+mL_0) & 0  \cdots  0 & (p_j+mL_0) - p_i & 0  \cdots  0  &   \cr 
},\]
where $\p_i, \p_j \in \mathbb R$, and $m \in \mathbb Z$. 
Since there is always a $1$-dimensional space of trivial infinitesimal motions generated by the vector $(1, \dots, 1)^T$, the rigidity matrix has maximum rank $|V|-1$. 

\begin{prop}
The periodic orbit framework $\pofw$ is (infinitesimally) rigid on $\Tor^1$ if and only if $G$ is connected.  
\label{prop:fixedCircle}
\end{prop}

If we allow the radius of the circle to change size, in addition to connectivity, we now require the graph to ``wrap" in a non-trivial fashion around the circle. That is, $\pog$ must contain a constructive cycle. The rigidity matrix now has an extra column corresponding to $x(t)$, with entry (for the edge $\{i,j; m\}$) given by $m(p_i - (p_j + mL))$. 

One way to see the necessity of a constructive cycle is to perform the $T$-gain procedure on the edges of a periodic orbit graph $\pog$ with $|E| = |V|$. If no cycle is constructive, the column corresponding to $x(t)$ will be identically zero. 

\begin{prop}
The periodic orbit framework $\pofw$ is (infinitesimally) rigid on $\T^1$ if and only if $G$ is connected, and $G$ contains a constructive cycle.  
\label{prop:flexCircle}
\end{prop}

\section{Necessary Conditions for Rigidity on $\Torx^2$}
\label{sec:necessary conds}

Let $\pog$ be a periodic orbit framework where $G$ is $P(2,1)$. We say $m$ is {\it $\Torx^2$-constructive} if 
\begin{enumerate}[i)]
	\item every $(2,2)$-subgraph is constructive (i.e. every $(2,2)$-subgraph contains some cycle with non-trivial net gain) and
	\item every $(2,1)$-subgraph is {\it $x$-constructive} (i.e. contains some cycle with non-trivial net gain in the $x$-direction.) 
\end{enumerate}

In \cite{LeeStreinu} it was shown that the class of $(k,\ell)$-tight graphs forms a matroid for natural numbers $0\leq \ell <2k$. Thus we say that a $(k,\ell-1)$-circuit is a $(k,\ell)$-tight graph in which deleting any edge gives a $(k,\ell-1)$-tight graph. Crucial to us will be the following weaker definition.

We will say that $G$ is a {\it $P(2,1)$-graph} if $G$ is $(2,1)$-tight and there exists $e \in E$ such that $G-e$ is $(2,2)$-tight.
The following proposition provides necessary conditions for minimal rigidity on $\Torx^2$. 

\begin{prop}\label{prop:necessary}
Let $\pofw$ be a periodic orbit framework. If $\pofw$ is minimally rigid on $\Torx^2$, then $G$ is $P(2,1)$ and $m$ is $\Torx^2$-constructive. 
\end{prop}

\begin{proof}
That $G$ must be $P(2,1)$ for $\pofw$ to be rigid follows from Theorem \ref{thm:treesAndMapsAreNec} and Lemma \ref{lem:oneFullyCounted}. Let $H$ be the unique $(2,2)$-circuit.


To see that $\pog$ is $\Torx^2$-constructive, we show that $H$ is $x$-constructive, and that every $(2,2)$-tight subgraph of $G$ obtained by deleting a single edge of $H$ is constructive.  

To see that $H$ is $x$-constructive, suppose toward a contradiction that $\pog$ contains no $x$-constructive cycle. Applying the $T$-gain procedure to any tree in $G$ will produce a $T$-gain assignment $m_T$, where all $x$-coordinates are zero. Then all of the entries of the single lattice column of the rigidity matrix will be zero. Hence we effectively have $2|V|-1$ edges in the fixed torus rigidity matrix, which has maximum rank $2|V| - 2$, a contradiction.

Since $G$ is $P(2,1)$, deleting any edge $e$ from $H$ results in a $(2,2)$-subgraph of $G$. Since every such subgraph must correspond to a set of linearly independent rows in the rigidity matrix, Proposition \ref{prop:depOnFixed} implies that the gain assignment $m$ restricted to this subgraph must be constructive. 

%

\end{proof}

When we move from the fixed torus to the variable torus, we add columns to the rigidity matrix. As a result, it seems possible that edges that were dependent on the fixed torus become independent on the variable torus.  When $d=2$, and our dependent subgraphs are of size $2|V|-2$,  this is not the case.
\begin{prop}
Let $\pog$ be a $(2,2)$-tight periodic orbit graph. If $\pog$ is dependent on $\Tor^2$ then $\pog$ is also dependent on $\T_x^2$. 
\label{prop:depOnFixed}
\end{prop}

\begin{proof}
Suppose that $\pog$ is dependent on $\Tor^2$. Then there is some subgraph $\pogp \subseteq \pog$ with $|E'| = 2|V'| - 2$, and no constructive cycle. Therefore, all gains on this subgraph are $T$-gain equivalent to $(0,0)$. Then the entries in the lattice column of the rigidity matrix corresponding to these edges will be zero, since $(m_T)_x=0$ for all $e$, and therefore the edges continue to be dependent on $\T_x^2$. %
\end{proof}

A {\it map-graph} is a graph in which each connected component has exactly one cycle. By a result of Whiteley \cite{UnionMatroids}, $G=(V,E)$ is a $(k, \ell)$-tight graph if and only if $E$ is the edge-disjoint union of $\ell$ spanning trees and $k-\ell$ spanning map-graphs. Note that each map-graph need {\it not} be connected. This is in contrast to the situation for minimally rigid periodic orbit frameworks on $\Torx^2$:

\begin{thm}
Let $\pofw$ be a minimally rigid framework on the variable torus $\Torx^2$. Then the edges of $\pog$ admit a decomposition into one spanning tree and one connected spanning map-graph. 
\label{thm:treesAndMapsAreNec}
\end{thm}

\begin{proof}
This proof is similar to the proof of Theorem 2.18 in \cite{PureCondition}. 
Let $\pofw$ be a minimally rigid framework on $\Torx^2$. The rigidity matrix, $\R_x\pofw$ has rank $2|V|-1$, and dimension $(2|V|-1) \times (2|V| + 1)$, with $2|V|$ columns corresponding to the vertices, and one column corresponding to the flexibility of the lattice. Adding the two rows 
$$\left(\begin{array}{cccccc}1 & 0 & 0 & \cdots &  0 & 0\end{array}\right) ,
\left(\begin{array}{cccccc}0 & 1 & 0 &\cdots & 0 & 0\end{array}\right)$$
has the effect of eliminating the $2$-dimensional space of infinitesimal translations. This ``tie down" is described in \cite{PureCondition}, and is equivalent to pinning one vertex on the torus. The resulting square matrix has $2|V| +1$ independent rows, and hence a non-zero determinant. 

Reorder the columns of $\R_x\pofw$ by coordinates, with the single lattice column grouped with the first coordinates. Regard the determinant as a Laplace decomposition where the terms are products of the determinants of one square block $M_i$ with dimension $(|V| +1) \times (|V|+1)$, and one square block $N_i$ with dimension $|V| \times |V|$. That is, $\det \R_x\pofw = \sum M_i N_i$. The block $M_i$ contains all of the entries from the columns of the first coordinates, and the $N_i$ block contains the second coordinates. Each block contains a single tie-down row. 

There must be at least one nonzero product $M_kN_k$. By the Laplace decomposition, the rows used in $M_k$ and $N_k$ form disjoint subgraphs, and $M_k$ or $N_k$ will each contain one of the tie-down rows. The $|V|$ rows of $M_k$ that are not tie-down rows have rank $|V|$, and this submatrix corresponds to the rigidity matrix of a $1$-dimensional graph on the flexible circle (the periodic line). By Proposition \ref{prop:flexCircle}, we know that such a graph must be connected, and must contain a constructive cycle. Hence the $|V|$ edges of the $M_k$ block form a spanning connected map-graph.  

Similarly, the $|V|-1$ edges  of the $N_k$ block that are not tie-down edges correspond to the rigidity matrix of a graph on the fixed circle. Since the block $N_k$ has rank $|V|$, the edges that are not tie-down edges are independent on the fixed circle, and hence by Proposition \ref{prop:fixedCircle} the graph is connected. Therefore the edges of the $N_k$ block form a spanning tree of $G$. %
\end{proof}

\begin{lem}
Suppose $G$ has $|E| = 2|V| - 1$. If the edges of $G$ admit a decomposition into one (edge-disjoint) spanning tree and one connected spanning map-graph,  then $G$ is a $P(2,1)$-graph and $G$ contains a unique $(2,2)$-circuit.
\label{lem:oneFullyCounted}
\end{lem}

\begin{proof}
In light of Theorem \ref{p21theorem}, the lemma is simply a re-statement of Lemma \ref{minimal}.
\end{proof}




\section{H1, H2 Preserve Rigidity on $\Torx^2$}
\label{sec:h1h2}

We will use $d(v)$ to denote the degree of the vertex $v$ and $d_G(v)$
when the context of the graph is not clear. $N(v)$ denotes the set of neighbours of $v$. As is common in the literature we will refer to the following construction moves as \emph{Henneberg operations}:

\begin{enumerate}
\item[(1a)] add a vertex $v_0$ with $d(v_0)=2$ and $N(v_0)=\{v_1, v_2\}$, $v_1 \neq v_2$, 
\item[(1b)] add a vertex $v_0$ with $d(v_0)=2$ and $N(v_0)=\{v_1\}$,
\item[(2a)] remove an edge $v_1v_2$, $v_1\neq v_2$, and add a vertex $v_0$ with $d(v_0)=3$ and $N(v_0)=\{v_1, v_2, v_3\}$ for some $v_3 \in V$,
\item[(2b)] remove an edge $v_1v_2$, $v_1\neq v_2$, and add a vertex $v_0$ with $d(v_0)=3$ and $N(v_0)=\{v_1, v_2\}$ with one edge connecting $v_0$ with $v_1$, and two edges connecting vertices $v_0$ and $v_2$.  
\end{enumerate}

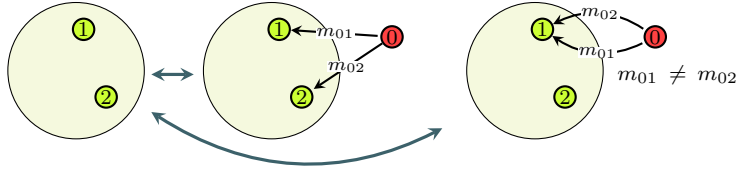
\begin{figure}[h!]
\begin{center}
\begin{tikzpicture}[->,>=stealth,shorten >=1pt,auto,node distance=2.8cm,thick, font=\footnotesize] 
\tikzstyle{vertex1}=[circle, draw, fill=couch, inner sep=.5pt, minimum width=3.5pt, font=\footnotesize]; 
\tikzstyle{vertex2}=[circle, draw, fill=melon, inner sep=.5pt, minimum width=3.5pt, font=\footnotesize]; 
\tikzstyle{voltage} = [fill=white, inner sep = 0pt,  font=\scriptsize, anchor=center];

	\draw (0.1,0.05) circle (.9cm); 
	\fill[cloud] (0.1,0.05) circle (.9cm);
		
	\node[vertex1] (1) at (.2,.6)  {$1$};
	\node[vertex1] (2) at (.5,-.3) {$2$};
	\path[<->, bluey, very thick] (1.1, -0.5) edge [bend right] (5, -.7);
	\pgftransformxshift{1.3cm}
	\draw[<->, very thick, bluey] (-.2, 0) -- (.4, 0);
	\pgftransformxshift{1.3cm}

	\draw (0.1,0.05) circle (.9cm); 
	\fill[cloud] (0.1,0.05) circle (.9cm);
		
	\node[vertex1] (1) at (.2,.6)  {$1$};
	\node[vertex1] (2) at (.5,-.3) {$2$};
	\node[vertex2] (3) at (1.7, .5) {$0$};

	\draw[thick, <-] (1) --  node[voltage] {$\bm_{01}$} (3);
		\draw[thick] (3) -- node[voltage] {$\bm_{02}$} (2);
					
	\pgftransformxshift{3.5cm}
	\draw (0.1,0.05) circle (.9cm); 
	\fill[cloud] (0.1,0.05) circle (.9cm);
		
	\node[vertex1] (1) at (.2,.6)  {$1$};
	\node[vertex1] (2) at (.5,-.3) {$2$};
		\node[vertex2] (3) at (1.7, .5) {$0$};

     \draw[thick, <-] (1) edge [bend right] node[voltage] {$\bm_{01}$} (3);
     \draw[thick,<-] (1) edge [bend left] node[voltage] {$\bm_{02}$} (3);

	\node[text width=2.5cm, text centered] at (2,0) {$\bm_{01} \neq \bm_{02}$};

\end{tikzpicture} 
\caption{The $H1$ moves (periodic vertex addition). The large circular region represents a generically rigid periodic orbit graph on $\Torx^2$.  \label{fig:vertexAddition}}
\end{center}
\end{figure}

\begin{figure}[h!]
\begin{center}
\begin{tikzpicture}[->,>=stealth,shorten >=1pt,auto,node distance=2.8cm,thick, font=\footnotesize] 
\tikzstyle{vertex1}=[circle, draw, fill=couch, inner sep=.5pt, minimum width=3.5pt, font=\footnotesize]; 
\tikzstyle{vertex2}=[circle, draw, fill=melon, inner sep=.5pt, minimum width=3.5pt, font=\footnotesize]; 
\tikzstyle{voltage} = [fill=white, inner sep = 0pt,  font=\scriptsize, anchor=center];

	\draw (0.1,0.05) circle (.9cm); 
	\fill[cloud] (0.1,0.05) circle (.9cm);
		
	\node[vertex1] (1) at (.2,.7)  {$1$};
	\node[vertex1] (2) at (.6,-.4) {$2$};
	   \node[vertex1] (4) at (-.4, -.1) {$3$};
	
	\draw[thick] (1) -- node[voltage, fill=cloud] {$\bm_e$} (2);
	\path[bluey, very thick] (1.1, -0.5) edge [bend right] (5.6, -.7);
	\pgftransformxshift{1.5cm}
	\draw[very thick, bluey] (-.2, 0) -- (.4, 0);
	\pgftransformxshift{1.5cm}

	\draw (0.1,0.05) circle (.9cm); 
	\fill[cloud] (0.1,0.05) circle (.9cm);
		
	\node[vertex1] (1) at (.2,.7)  {$1$};
	\node[vertex1] (2) at (.6,-.4) {$2$};
	   \node[vertex1] (4) at (-.4, -.1) {$3$};
	\node[vertex2] (3) at (1.7, .5) {$0$};

	\draw[thick] (1) --  node[voltage] {$(0,0)$} (3);
		\draw[thick] (3) -- node[voltage] {$\bm_e$} (2);
		\draw[thick] (3) -- node[voltage] {$\bm_{03}$} (4);
		
	\pgftransformxshift{3.5cm}
	\draw (0.1,0.05) circle (.9cm); 
	\fill[cloud] (0.1,0.05) circle (.9cm);
		
	\node[vertex1] (1) at (.2,.7)  {$1$};
	\node[vertex1] (2) at (.6,-.4) {$2$};
	   \node[vertex1] (4) at (-.4, -.1) {$3$};
	\node[vertex2] (3) at (1.7, .5) {$0$};

	\draw[thick] (1) -- node[voltage] {$(0,0)$} (3);
	\draw[thick] (3) edge [bend right]  node[voltage] {$\bm_{03}$} (2);
	\draw[thick] (3) edge [bend left] node[voltage] {$\bm_{e}$} (2);

	\node[text width=2.5cm, text centered] at (2.5,-0.75) { $\bm_{03} \neq \bm_{e}$};
\end{tikzpicture}
\caption{The $H2$ moves (periodic edge split). The gain $m_e$ on the edge connecting 1 and 2 is preserved through this split. \label{fig:edgeSplit}}
\end{center}
\end{figure}
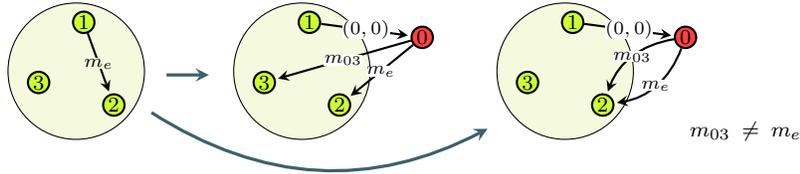

More strongly,  when the moves are applied to a periodic orbit framework $\pog$ with changes to the gains as illustrated in Figures \ref{fig:vertexAddition} and \ref{fig:edgeSplit}, we say that these operations are \emph{gain-preserving Henneberg operations}.

It was shown in \cite{ThesisPaper2} using linear algebra techniques that gain-preserving Henneberg operations preserve the maximality of the rank of the rigidity matrix, echoing the situation for finite frameworks. The corresponding results for the variable torus $\Torx^2$ can be proven entirely similarly; we leave the details to the reader.

\begin{prop}\label{prop:h1rigidity}
Let $\pog$ be a periodic orbit and let $\langle G',m' \rangle$ be the result of a gain-preserving $H1$ operation on $\pog$. Let $p$ be generic and let $p'=(p,p_{n+1})$ be chosen generically with respect to $p$. Then the rows of $\R_x(\pog, p)$ are linearly independent if and only if the rows of $\R_x(\langle G',m'\rangle,p')$ are linearly independent.
\end{prop}

\begin{prop}\label{prop:h2rigidity}
Let $\pog$ be a periodic orbit and let $\langle G',m' \rangle$ be the result of a gain-preserving $H2$ operation on $\pog$. Let $p$ be generic and let $p'=(p,p_{n+1})$ be chosen generically with respect to $p$. If the rows of $\R_x(\pog, p)$ are linearly independent then the rows of $\R_x(\langle G',m'\rangle,p')$ are linearly independent.
\end{prop}


\section{$P(2,1)$ Graphs}
\label{sec:p21}

In this section we consider in detail the structure of $P(2,1)$-graphs.  For $X\subset V$, $G[X]$ denotes the subgraph
induced by $X$.
For two subsets $J,K \subset V$, $d(J,K)$ denotes the number of edges in $E$ with one end-vertex in $G[J]$ and one in 
$G[K]$. For a subset $X\subset V$ 
let $i(X)$ denote the number of edges in the subgraph induced by $X$.
Observe first that a $P(2,1)$-graph can have at most one loop and a $(2,2)$-circuit can have a loop if and only if the graph is a single loop.

\subsection{Critical Sets}
\label{criticalsubsec}

Let $G=(V,E)$ be a $P(2,1)$-graph and let $X\subset V$. We will say $X$ is \emph{over-critical} if $i(X)=2|X|-1$, \emph{critical} if
$i(X)=2|X|-2$ and \emph{semi-critical} if $i(X)=2|X|-3$. In each case if there is a degree $3$ vertex $v\in V$ with $N(v)=\{x,y,z\}$ and $X$
contains $x,y$ but not $z,v$ then we say that $X$ is \emph{over-v-critical, v-critical} or \emph{semi-v-critical} respectively.

To simplify the arguments in Section \ref{gainsection} we record some basic facts about $P(2,1)$-graphs.
In the next lemma, minimal means having the least number of vertices.

\begin{lem}\label{minimal}
Let $G=(V,E)$ be a $P(2,1)$-graph. Then there is a unique minimal over-critical set $X\subset V$ and $G[X]$ is a $(2,2)$-circuit.
\end{lem}

\begin{proof}
If $G$ is a $(2,2)$-circuit, then $X=V$. Otherwise by definition there exists $X\subsetneq V$ with $i(X)=2|X|-1$.  
Choose the minimal over-critical set $X^-\subset X$. This induces a $(2,2)$-circuit.

For uniqueness suppose $J\subsetneq V$ such that $J\neq X^-$ is over-critical. If $J$ and $X^-$ are disjoint we contradict the definition of
a $P(2,1)$-graph. Thus $i(X^- \cup J)+i(X^-\cap J)=i(X^-)+i(J)+d(X^--J,J-X^-)=2|X^-|-1+2|J|-1+d(X^--J,J-X^-)$. 
This implies that $d(X^--J,J-X^-)=0$, $i(X^- \cup J)=2|X^-\cup J|-1$ and $i(X^-\cap J)=2|X^-\cap J)-1$ contradicting the minimality of
$X^-$.
\end{proof}

\begin{lem}\label{connectedlem}
Let $G=(V,E)$ be a $P(2,1)$-graph. Then 
\begin{enumerate}[(1)]
\item $G$ is $2$-edge-connected. 
\item If $X\subset V$ is critical then $G[X]$ is connected.
\item If $X\subset V$ is semi-critical then either $G[X]$ is connected or $X$
has two connnected components $A,B$ such that $A$ is over-critical and $B$ is critical.
\end{enumerate}
\end{lem}

\begin{proof}
In each case Lemma \ref{minimal} will imply at most one of $A,B$ is over-critical.

For (1) suppose $V=A\cup B$ for $A,B \subset V$ with $A\cap B=\emptyset$ and $d(A,B)=1$.
Now
\begin{eqnarray*}2|V|-1= |E|&=&i(A)+i(B)+1 \\ &\leq& 2|A|-1+2|B|-2+1\\ &=&2|V|-2,\end{eqnarray*}
a contradiction.

For (2) suppose $V=A\cup B$ for $A,B \subset V$ with $A\cap B=\emptyset$ and $d(A,B)=0$. Now
\begin{eqnarray*}2|X|-2= i(X)&=&i(A)+i(B) \\ &\leq& 2|A|-1+2|B|-2\\ &=&2|X|-3,\end{eqnarray*}
a contradiction.

For (3) suppose $G[X]$ is not connected. Then the proof is entirely similar to (2).
\end{proof}

However $P(2,1)$-graphs need not be $2$-connected and need not be $3$-edge-connected. 

\begin{lem}\label{criticallem1}
Let $G=(V,E)$ be a $P(2,1)$-graph with unique $(2,2)$-circuit $G[X]$ and let $v\in V$ have $N(v)=\{x,y,z\}$.
Let $x,y,z\in X$. Then 
\begin{enumerate}[(1)]
\item $v\in X$ and there is no over-v-critical set.
\item If $Y$ is v-critical on $x,y$ but not $z$ with no over-critical subset and $Z$ is semi v-critical on $x,z$ but not $y$ with no
critical subset then
\begin{enumerate}[(a)]
\item if $|Y\cap Z|>1$ then $Y\cup Z$ is critical and $Y\cap Z$ is semi-critical.
\item if $|Y\cap Z|=1$ then $Y\cup Z$ is semi-critical and $Y\cap Z$ is critical.
\end{enumerate}
\end{enumerate}
\end{lem}

\begin{proof}
First $v \in X$ otherwise $i(X\cup v)=2|X\cup v|$ and any over-v-critical set contradicts Lemma \ref{minimal}.

For (2) $Y\cap Z \subset Z$ so $i(Y\cap Z)\leq 2|Y\cap Z|-3$. Since 
\[ i(Y\cup Z)+i(Y\cap Z)=2|Y\cup Z|+2|Y\cap Z|-5+d(Y,Z)\]
and $Y\cup Z$ contains $x,y,z$ but not $v$ we deduce that $Y\cup Z$ is critical and $Y\cap Z$ is semi-critical.
(b) is entirely similar.
\end{proof}

\begin{lem}\label{criticallem2}
Let $G=(V,E)$ be a $P(2,1)$-graph with unique $(2,2)$-circuit $G[X]$ and let $v\in V$ have $N(v)=\{x,y,z\}$.
Let $x,y \in X, z \in V-X$. Then 
\begin{enumerate}[(1)]
\item $v \in V-X$ and $X$ is the unique over-v-critical set, 
\item there is no v-critical set containing $z$,
\item if $Y\subset V$ contains $x,z$ but not $y,v$ and $|X\cap Y|>1$ then $G[Y]$ is not $(2,3)$-tight.
\end{enumerate}
\end{lem}

\begin{proof}
First $v \in V-X$ contradicts Lemma \ref{minimal} and $X$ is unique by Lemma \ref{minimal} and the definition of a $(2,2)$-circuit.

For (2) suppose there is such a v-critical set $Z$. 
\begin{eqnarray*} i(X\cup Z)+i(X\cap Z)&=&i(X)+i(Z)+d(X-Z,Z-X)\\ &\leq& 2|X|-1+2|Z|-2+d(X-Z,Z-X). \end{eqnarray*}
Since $i(X\cap Z)\leq 2|X\cap Z|-2$ we have $d(X-Z,Z-X)=0$ and $i(X\cup Z)=2|X\cup Z|-1$ but $i(X\cup Z\cup v)=2|X\cup Z\cup v|$, a
contradiction.

Finally, since $|X\cap Y|>1$ we have $i(X\cap Y)\leq 2|X\cap Y|-3$ as $X\cap Y\subset Y$. Thus, similarly to before, $X\cup Y$ is over 
v-critical, and adding back $v$ gives a contradiction.
\end{proof}

\begin{lem}\label{criticallem3}
Let $G=(V,E)$ be a $P(2,1)$-graph with unique $(2,2)$-circuit $G[X]$ and let $v\in V$ have $N(v)=\{x,y,z\}$.
Let $y,z \in V-X$. Then $v \in V-X$, there is at most one over-v-critical set $Y$. Moreover if there is such a $Y$ then there is no
v-critical set and at most one semi-v-critical set.
\end{lem}

\begin{proof}
First $v \in V-X$ otherwise $i(X-v)>2|X-v|-1$. Let $Y$ be over-v-critical containing $x,z$ but not $y$. If $W$ is over-v-critical on
$x,y$ but not $z$ then $i(Y\cup W\cup v)\geq 2|Y\cup W\cup v|$.

Now suppose $Z$ is v-critical on $x,y$ but not $z$. Then $i(Y\cap Z)\leq 2|Y\cap Z|-2$ so $i(Y\cup Z)=2|Y\cup Z|-1$ and adding back $v$
creates a contradiction.

Finally suppose $Z$ is semi-v-critical on $x,y$ but not $z$. Then $Y\cup Z$ is critical. Consider $W\subset V$ containing $y,z$ but not
$x,v$. $i(Y\cap W)\leq 2|Y\cap W|-2$ and $i(Z\cap W)\leq 2|Z\cap W|-2$. Therefore $i(Y\cup Z\cup W)\leq 2|Y\cup Z \cup W|-4$.
\end{proof}

\begin{lem}\label{semicriticallemma}
Let $G=(V,E)$ be a $P(2,1)$-graph containing a unique $(2,2)$-circuit $G[X]$, with a degree $3$ vertex $v$ with $N(v)=\{x,y,z\}$ such that
$x,y \in X, z \notin X$. Let $Y_{xz}\subset V$ contain $x,z$ but not $y,v$ and $Y_{yz}\subset V$ contain $y,z$ 
but not $x,v$. Then at most one of $Y_{xz}$ and $Y_{yz}$ is semi-critical.
\end{lem}

\begin{proof}
Suppose $Y_{yz}$ is semi-critical.
Firstly if $i(Y_{xz})<2|Y_{xz}|-3$ then $i(X\cap Y_{xz})\geq 2|X\cap Y_{xz}|- 2$ by Lemma \ref{minimal}. Thus 
$i(X\cup Y_{xz})<2|X\cup Y_{xz}|-2$ so adding back $v$ and its $3$ edges contradicts the definition of a $P(2,1)$-graph.
Suppose $i(Y_{xz})=2|Y_{xz}|-3$. 
\begin{eqnarray*}i(X\cup Y_{xz})+i(X\cap Y_{xz})&=&i(X)+i(Y_{xz})+d(X-Y_{xz},Y_{xz}-X)\\ &=&4+d(X-Y_{xz},Y_{xz}-X).\end{eqnarray*}
If $i(X\cup Y_{xz})=2|X\cup Y_{xz}|-1$ add $v$ and its 
$3$ edges for a contradiction. If $i(X\cap Y_{xz})=2|X\cap Y_{xz}|-1$ then we contradict $G[X]$ being a $(2,2)$-circuit. Thus 
$i(X\cup Y_{xz})=2|X\cup Y_{xz}|-2$ and $i(X\cap Y_{xz})=2|X\cap Y_{xz}|-2$. 
Since $X\cap Y_{xz} \subset Y_{xz}$ we know $|X\cap Y_{xz}|=1$. Similarly we derive that $|X\cap Y_{yz}|=1$.
Now 
\[i(Y_{xz}\cup Y_{yz})+i(Y_{xz}\cap Y_{yz})=6+d(Y_{xz}-Y_{yz},Y_{yz}-Y_{xz}).\]
$i(Y_{xz}\cap Y_{yz})>2|Y_{xz}\cap Y_{yz}|-1$ by Lemma \ref{minimal} and $i(Y_{xz}\cup Y_{yz})>2|Y_{xz}\cup Y_{yz}|-1$ otherwise
adding $v$ and its $3$ edges gives a contradiction. Hence
\[i(Y_{xz}\cup Y_{yz})=2|Y_{xz}\cup Y_{yz}|-s\] 
where $s \in \{2,3,4\}$.
\[i(X\cup (Y_{xz}\cup Y_{yz}))=1+s-4 +d(X-(Y_{xz}\cup Y_{yz}),(Y_{xz}\cup Y_{yz})-X)\leq 1.\] 
Thus adding $v$ and its $3$ edges violates the definition of a $P(2,1)$-graph. 
%
%
%
\end{proof}


\subsection{Henneberg Operations on $P(2,1)$-graphs}

A vertex $v$ in a $P(2,1)$-graph is \emph{admissible} 
if there is some inverse Henneberg operation removing $v$ that results in a $P(2,1)$-graph.

For brevity we will use $K_i^j$ to denote the complete graph on $i$ vertices with $j$ copies of each edge. Note
when $i=1$ then $j$ denotes the number of loops on that single vertex.

The next lemma is useful in extending the standard arguments showing the inverse moves preserve 
$(2,l)$-tightness to show that the inverse moves preserve the $P(2,1)$ condition.

\begin{lem}\label{suitableedgelemma}
Let $G=(V,E)$ be a $P(2,1)$-graph not equal to $K_2^3$ with no degree $2$ vertex. Then there exists a degree
$3$ vertex $v$ and an edge $e$ not incident to $v$ such that $G-e$ is $(2,2)$-tight.
\end{lem}

\begin{proof}
This is trivial if $G$ is a $(2,2)$-circuit. Suppose $G$ is not a $(2,2)$-circuit then there exists an over-critical set $K \subsetneq V$. 
Let $J=V-K$. Now $i(V)=i(J)+i(K)+d(J,K)$ so $i(J)=d(J,K)$. Also
$i(J)=d(J,K)\geq 2$ by Lemma \ref{connectedlem} (1).

Suppose $d_G(v)\geq 4$ for all $v \in J$. 
There are at least $4|J|+2$ vertex/edge incidences which implies $i(J)+d(J-K,K-J) \geq 2|J|+1$ but 
\[ 2|V|-1=i(V)= i(K)+i(J)+d(J-K,K-J)\geq 2|K|+2|J|\]
so there is a vertex $v_3 \in J$ with $d_G(v_3)=3$. Now from the definition of a $P(2,1)$-graph there is some edge in $G[K]$ 
that gives the result.
\end{proof}

%
%
%
%

The following lemma allows us to apply the inverse moves easily in the case that the $P(2,1)$-graph $G$
happens to be a $(2,2)$-circuit.

\begin{lem}\label{deletedegree3prop}
Let $G=(V,E)$ be a graph. If $G$ is a $(2,2)$-circuit then $G$ contains a degree $3$ vertex $v$, and
for any such $v$, $G-v$ is $(2,2)$-tight. 
\end{lem}

Note that the lemma is not in general true for $P(2,1)$-graphs and that the converse fails for 
$(2,2)$-circuits but is true by definition for $P(2,1)$-graphs.

\begin{proof}
Let $G$ be a $(2,2)$-circuit. Since $|E|=2|V|-1$ we have $\sum_{i=1}^{|V|}(4-d(i))=2$. Therefore there exists $v \in V$ with 
$d(v) \leq 3$. $(2,2)$-circuits cannot contain vertices of degree $\leq 2$; suppose $u$ was such a vertex then 
$i(V- u)=2|V- u|-1$ edges contradicting the definition of a $(2,2)$-circuit. 
Thus $G$ contains a degree $3$ vertex $v$. Clearly $V \setminus v$ is critical since we can think of the
operation of deleting a degree $3$ vertex as the composition of an edge deletion and an inverse 1a move.
\end{proof}

\begin{lem}\label{circuitinversecor}
Let $G$ be a $(2,2)$-circuit. Then either $G=K_2^3$ or there is an inverse 2a or 2b move on any degree $3$ vertex that results in a 
$P(2,1)$-graph.
\end{lem}

\begin{proof}
Let $v \in V$ have $d(v)=3$. $V-v$ is critical by Lemma \ref{deletedegree3prop}.
If $N(v)=\{a\}$ then $v \in K_2^3$. If not then adding any non-edge between the neighbours of $v$ creates a $P(2,1)$-graph.
\end{proof}

Note the stronger statement that every degree $3$ vertex in a $(2,2)$-circuit is admissible is false, for similar considerations see \cite{BergJordan} and \cite{Nixon}.

In the following lemma, recall it is well known that if $G$ is $(2,\ell)$-tight ($\ell=2,1$) then there is an inverse 2a or 2b move that results
in a $(2,\ell)$-tight graph. Hence we need only concern ourselves with the additional subgraph condition.

\begin{lem}\label{2a2binverse}
Let $G=(V,E)$ be a $P(2,1)$-graph containing a vertex $v$ with $d(v)=3$ and $|N(v)|>1$. 
Then $v$ is admissible.
\end{lem}

\begin{proof}
%
By Lemma \ref{circuitinversecor} we may assume $G$ is not a $(2,2)$-circuit. Thus by Lemma \ref{minimal} $G$ contains a unique 
$(2,2)$-circuit $G[X]$. By Lemma \ref{suitableedgelemma} we find a vertex $v$ in $V-X$ with $d(v)=3$.
For any edge $e$ in $G[X]$, $G-e$ is $(2,2)$-tight and $G-e-v+xy$ for some $x,y \in N(v)$ is $(2,2)$-tight. Therefore $G-v+xy$ is a
$P(2,1)$-graph.

\end{proof}

For brevity we use the terminology \emph{leaf} for a degree $1$ vertex.
We can now collect together the results in this section into our main result about $P(2,1)$-graphs.
$(1) \Leftrightarrow (3)$ is a slight refinement of a result of Whiteley. 

\begin{thm}\label{p21theorem}
Let $G=(V,E)$. The following are equivalent:
\begin{enumerate}
\item $G$ is a $P(2,1)$-graph,
\item $G$ can be constructed from $K_1^1$ or $K_2^3$ by 1a, 1b, 2a and 2b moves,
\item $G$ is the edge disjoint union of a spanning tree $T$ and a connected map graph $M$.
\end{enumerate}
\end{thm}

\begin{proof}
First we prove $(1) \Leftrightarrow (2)$.
It is easy to see that any of these four moves applied to an arbitrary $P(2,1)$-graph results in a $P(2,1)$ 
graph. Since $K_1^1$ and $K_2^3$ are $P(2,1)$-graphs it follows that any graph constructed from a sequence
of these moves is a $P(2,1)$-graph.

The converse follows from the above sequence of results by induction on $|V|$. Suppose $G$ is a $P(2,1)$-graph containing a loop.
Then this loop is the unique $(2,2)$-circuit within $G$ and Lemma \ref{2a2binverse} guarantees an inverse move.
Suppose now that $G$ is loopless. Then either $G$ is a $(2,2)$-circuit in which case apply Lemma \ref{circuitinversecor} or $G$ contains 
a unique $(2,2)$-circuit and apply Lemma \ref{suitableedgelemma}. The result follows from Lemma \ref{2a2binverse}.

$(3) \Rightarrow (1)$ follows since $|E|=|E(T)|+|E(M)|$, $|E(T)|=|V|-1$, $E(T')|\leq |V(T')|-1$ for any
subgraph $T'$ of $T$, $|E(M)|=|V|$ and $|E(M')|\leq |V(M')|$ for any subgraph $M'$ of $M$.

Let $G'$ be formed from $G$ by one of the four construction moves. $(2) \Rightarrow (3)$ follows by showing
that in each case if $G$ satisfies $(3)$ then so does $G'$. This is trivial in each case. Particularly in the
1a and 1b moves the new vertex is a leaf in the tree and in the connected map graph. In the 2a and 2b moves removing $xy$ the 
new edges $xv,yv$ go in whichever of $T$ or $M$ contained $xy$ and the remaining edge goes in the other.
\end{proof}



\section{$G$ is $P(2,1)$ and $\Torx^2$-constructive $\iff$ H1, H2}
\label{gainsection}

Let $\pog$ be a $\Torx^2$ orbit graph, where $G$ is $P(2,1)$ and $m$ is $\Torx^2$-constructive. We say that a vertex $v \in V$ is a {\it circuit vertex} if it is contained within the minimal $(2,1)$-tight subgraph of $G$ (this subgraph is a $(2,2)$-circuit). 
Recall the Henneberg operations defined in Section \ref{sec:h1h2}.
As we will see there is an infinite but controllable class of periodic orbit frameworks which are $P(2,1)$-graphs and $\Torx$-constructive, yet for which the Henneberg operations are insufficient, see Figure \ref{fig:bunnyEars}. For these graphs we must introduce an additional $H2$ Henneberg type move, see Figure \ref{fig:bunnyEarsH3} for the definition. In the following theorem, our main result, by Henneberg operation we mean $H1$ or $H2$ move.

\begin{thm}[$\Torx^2$ Henneberg Theorem]\label{thm:mainresult}
Let $\pog$ be a periodic orbit framework. $G$ is a $P(2,1)$-graph and $m$ is $\Torx$-constructive if and only if $G$ can be generated from a single loop by gain-preserving Henneberg operations.
\end{thm}


Given Theorem \ref{thm:mainresult} we may now prove Theorem \ref{thm:rigidityinduction}. 

\begin{proof}[Proof of Theorem \ref{thm:rigidityinduction}]
The necessity of the construction operations follows from Proposition \ref{prop:necessary} and the fact that each of the construction operations takes a $\Torx^2$-constructive $P(2,1)$-graph to a $\Torx^2$-constructive $P(2,1)$-graph.

By Theorem \ref{thm:rigidiffrank}, Propositions \ref{prop:h1rigidity} and \ref{prop:h2rigidity} show that periodic rigidity is preserved by the Henneberg operations. Thus the sufficiency follows from Theorem \ref{thm:mainresult}.
\end{proof}

The remainder of this section will prove Theorem \ref{thm:mainresult}.

\subsection{Paths, Cycles and Gains}

\begin{lem}\label{pathslem1}
Let $\pog$ be a periodic orbit framework where $G=(V,E)$ is a $P(2,1)$ graph with unique minimal over-critical set $X\subset V$ and
$m$ is $\Torx$-constructive. Let $v_0\in V$ have $N(v_0)=\{v_1,v_2,v_3\}\subset X$. Let $G_1^1$ be $(2,2)$-tight containing $v_1,v_2$ but
not $v_0,v_3$ with no $x$-constructive cycle and all paths from $v_1$ to $v_2$ have $x$-gain $(m_2-m_1)_x$. Similarly let $G_2^1$ be 
$(2,2)$-tight containing $v_2,v_3$ but not $v_0,v_1$ with no $x$-constructive cycle and all paths from $v_1$ to $v_2$ have $x$-gain 
$(m_3-m_2)_x$. Then all paths from $v_1$ to $v_3$ have $x$-gain $(m_3-m_1)_x$.
\end{lem}

\begin{proof} 
Let $P$ be a path from $v_1$ to $v_3$. The path must pass through $G_1^1 \cap G_2^1$, say through the vertex $u \in V_1^1 \cap V_2^1$. See Figure \ref{fig:pathsLem}.
In the simplest case, the path $v_1 \rightarrow u$ lies completely within $G_1^1$, and the path $u \rightarrow v_3$ lies completely 
within $G_2^1$. Since the graph $G_1^1 \cap G_2^1$ is connected, there exists a path $P' \in G_1^1 \cap G_2^1$ from $u \rightarrow v_2$. 
Suppose the gains on the path are as follows: 
\[v_1 \xrightarrow{m_a} u \xrightarrow{m_b} v_3\]
\[v_1 \xrightarrow{m_a} u \xrightarrow{m_c} v_2 \xrightarrow{-m_c} u \xrightarrow{m_b} v_3.\]
But it is now clear that 
\[(m_a + m_c)_x = (m_2 - m_1)_x, \textrm{ and \ } (m_b - m_c)_x = (m_3 - m_2)_x.\] 
Summing the two equations, we find that $m_a + m_b = m_3 - m_1$, as desired. 
\begin{figure}[htbp]
\begin{center}
\includegraphics[width=2.5in]{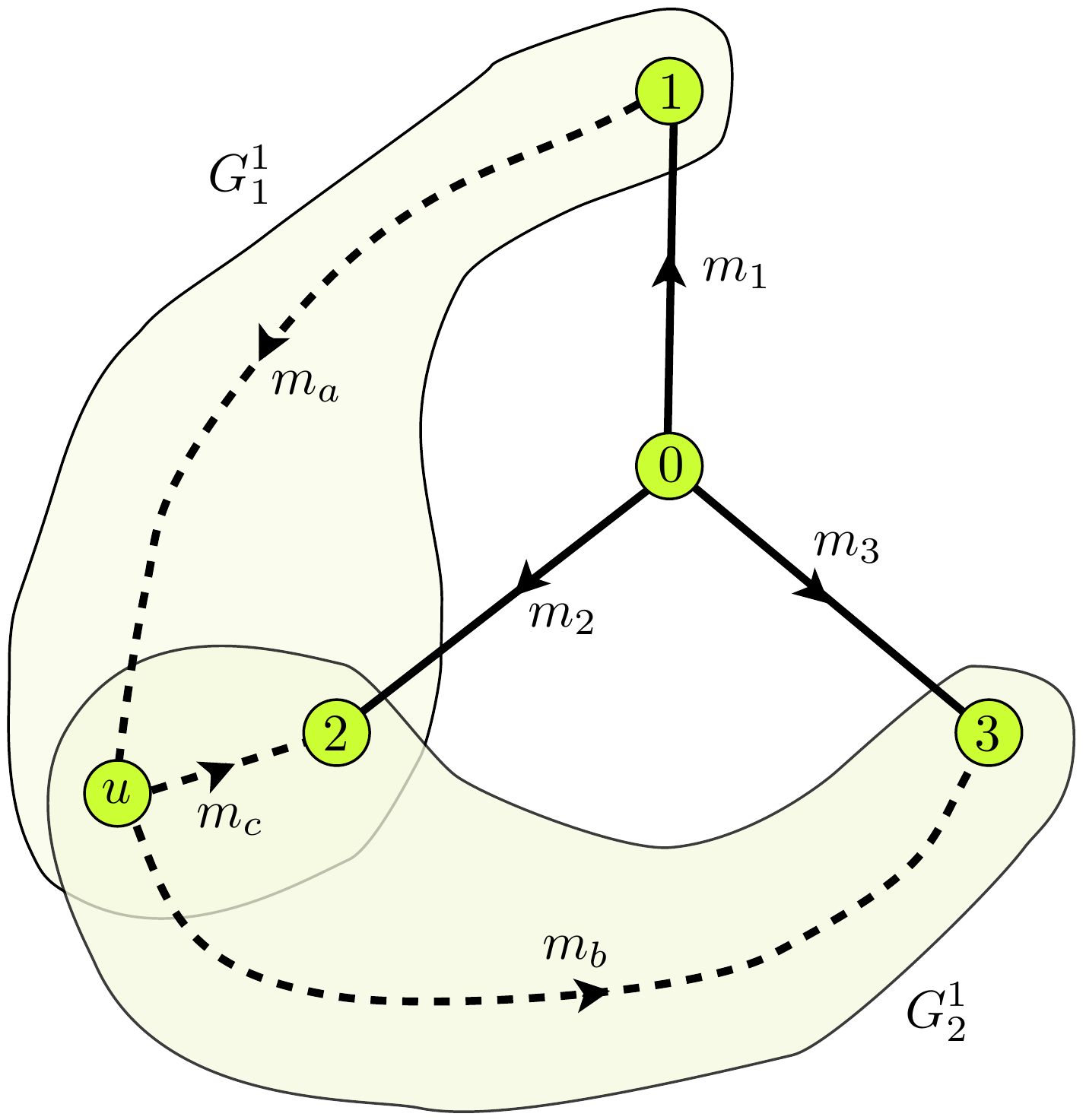}
\caption{Proof of Lemma \ref{pathslem1}. Both $G_1^1$ and $G_2^1$ are (2,2)-tight.}
\label{fig:pathsLem}
\end{center}
\end{figure}

Now consider the case that the path $P$ is as follows: 
\[v_1 \rightarrow u_1 \rightarrow u_2 \rightarrow \cdots \rightarrow u_n \rightarrow v_3,\]
where $v_1 \rightarrow u_1$ is contained within $G_1^1$, the path $u_1 \rightarrow u2 \in G_2^1$, and the remaining segments alternate between $G_1^1$ and $G_2^1$, with the final path  $u_n \rightarrow v_3 \in G_2^1$. 

Note first that $v_2 \rightarrow u_1 \rightarrow u_2 \rightarrow v_2$ is completely contained within $G_2^1$, and therefore it must have trivial net $x$-gain. Similarly, the net $x$-gain on any closed path 
\begin{equation}\underbrace{v_2 \rightarrow u_i \rightarrow u_{i+1} \rightarrow v_2}_{0} \label{eqn:pathslem1} \end{equation}
is trivial. Then the net $x$-gain on the path from $v_1$ to $v_3$ is as follows:
\[v_1 \rightarrow u_1 \rightarrow u_2 \rightarrow \cdots \rightarrow u_n \rightarrow v_3\]
\[\underbrace{v_1 \rightarrow u_1 \rightarrow v_2}_{m_2-m_1} \rightarrow u_1 \rightarrow u_2 \rightarrow v_2 \rightarrow \cdots v_2 \rightarrow u_{n-2} \rightarrow u_{n-1} \rightarrow \underbrace{v_2 \rightarrow u_n \rightarrow v_3}_{m_3-m_2},\]
and all the paths in the middle contribute nothing to the net $x$-gain by (\ref{eqn:pathslem1}), which completes the proof. 
\end{proof}

Note that the only graph theoretical consideration in the proof was the fact that $G_1^1 \cap G_2^1$ is connected. Thus, in view of
the results in Subsection \ref{criticalsubsec}, the argument adapts easily to each of the other cases we will consider.

\subsection{Two Distinct Neighbours}

In approaching this case, we need to pay special attention to a particular class of graphs which we call the {\it bunny ears} class. It consists of a three-valent vertex $v_0$, which is adjacent only to the vertices $v_1$, and $v_2$. The vertices $v_1$ and $v_2$ are both members of two edge-disjoint $(2,3)$-tight subgraphs, which comprise the rest of the graph (see Figure \ref{fig:bunnyEars}(a)). The simplest example, which gives the class its name, consists of only five edges, as shown in Figure \ref{fig:bunnyEars}(b).  

\begin{figure}[htbp]
\begin{center}
\subfloat[]{\includegraphics[width=2.5in]{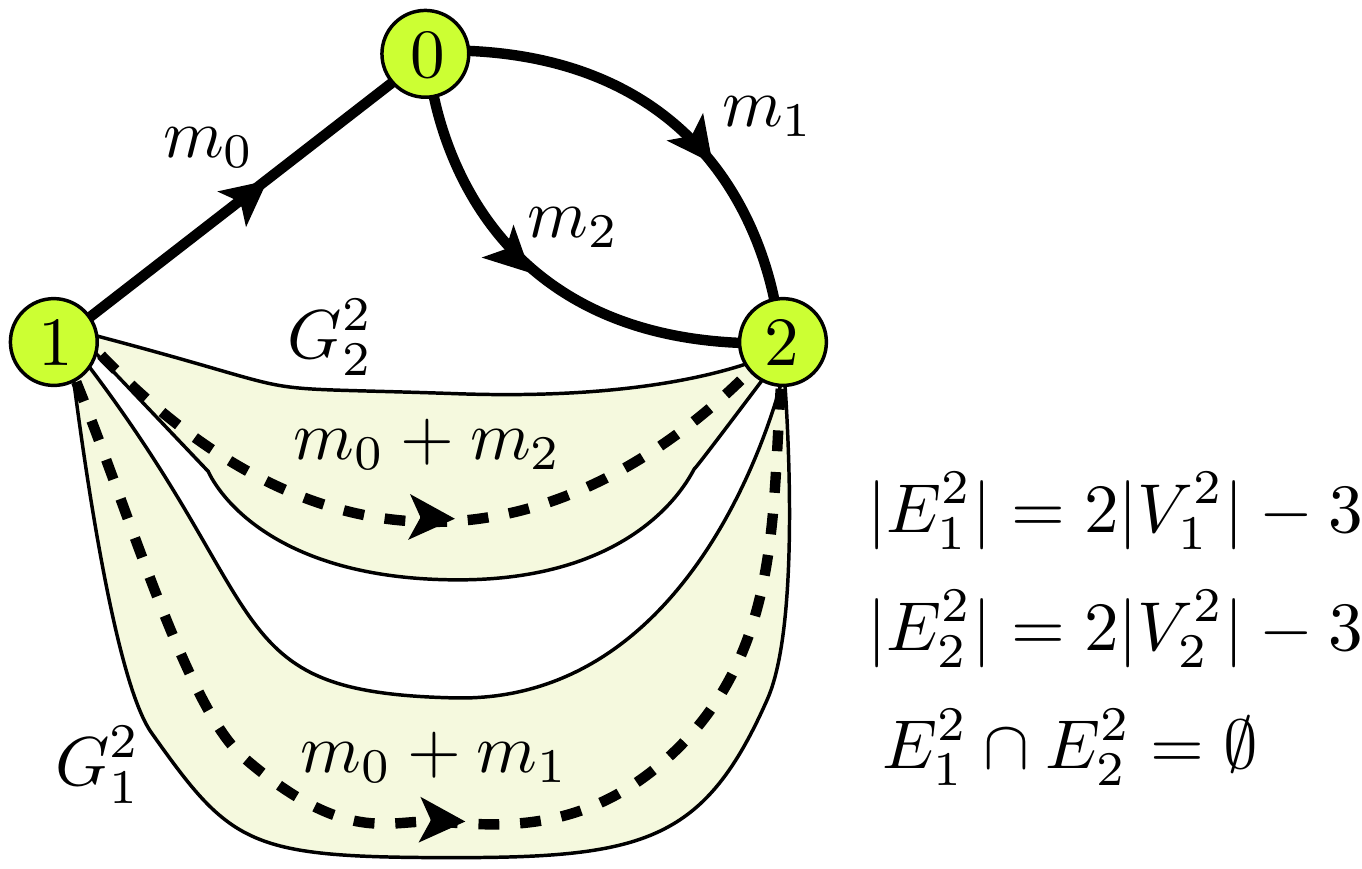}}\hspace{.25in}
\subfloat[]{\includegraphics[width=1.75in]{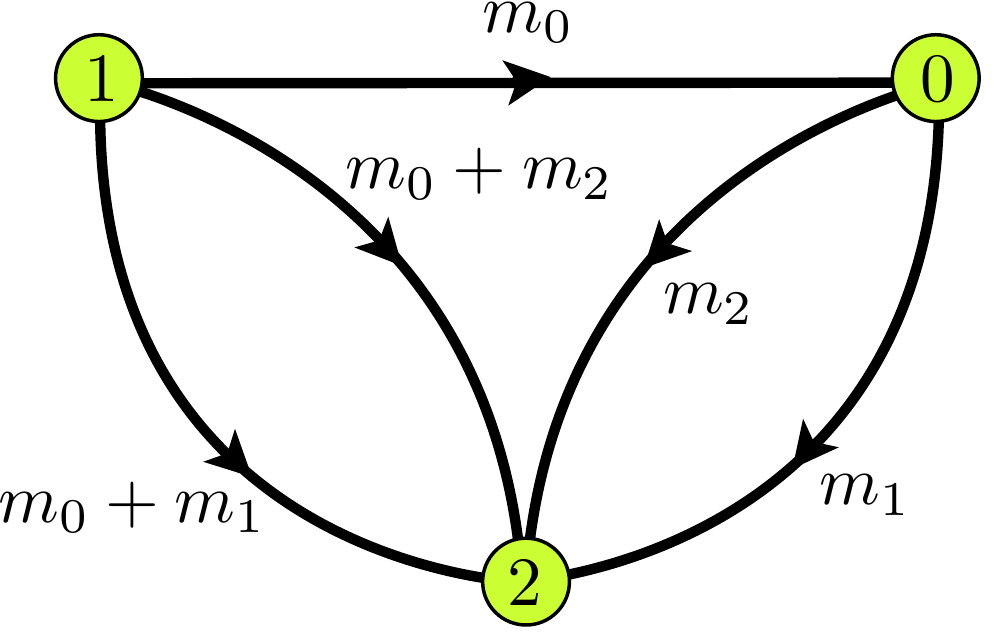}} 
\caption{The bunny ears class of graphs (a). The smallest example, where $G_1^2$ and $G_2^2$ consist of single edges is shown in (b). Note that attempting to delete either vertex $0$ or $1$ in a reverse $H2$ move will produce a dependent $(2,1)$-tight subgraph. We define a Henneberg 2c move to deal with this case, see Figure \ref{fig:bunnyEarsH3}.}
\label{fig:bunnyEars}
\end{center}
\end{figure}

For this class of graphs only, we introduce one additional Henneberg move ({\it Henneberg 2c}), which is essentially an edge split on a loop edge. We will use the reverse move: deleting a three-valent vertex and adding a loop edge. See Figure \ref{fig:bunnyEarsH3}. 

\begin{figure}[htbp]
\begin{center}
\subfloat[]{\includegraphics[width=1.75in]{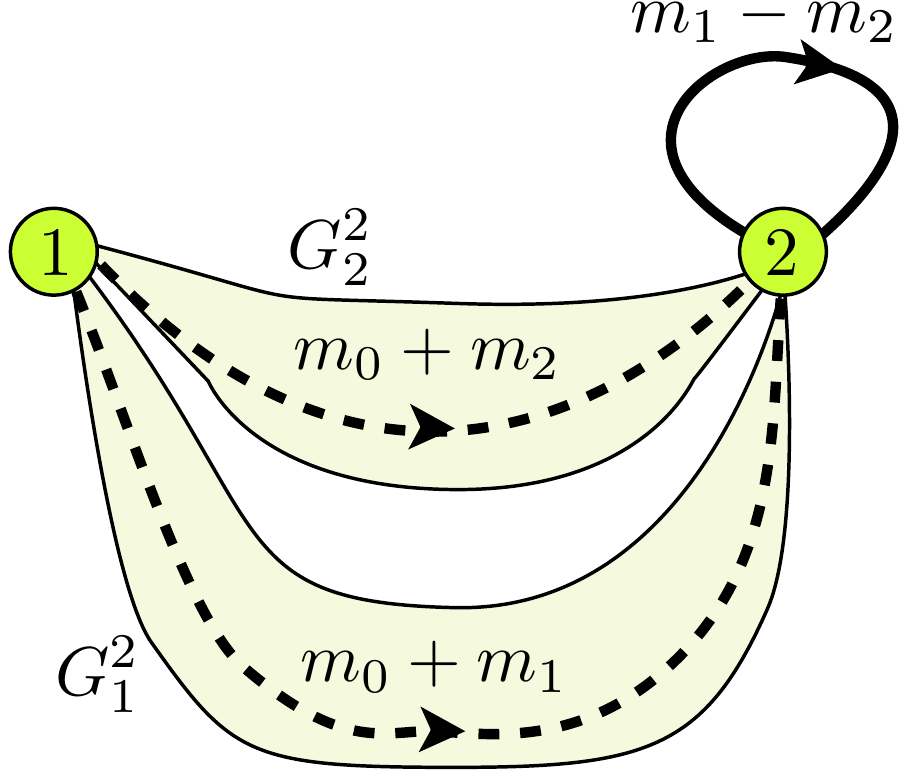}}\hspace{.5in}
\subfloat[]{\includegraphics[width=2in]{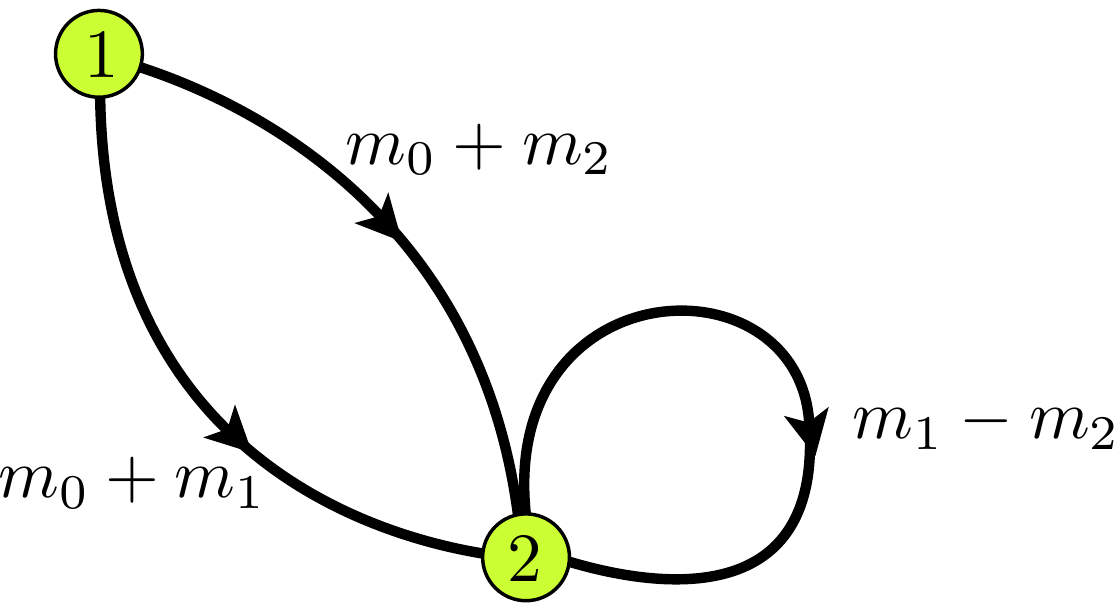}} 
\caption{The reverse Henneberg 2c move on the bunny ears class of graphs (a). The reverse Henneberg 2c move on the smallest example (b). }
\label{fig:bunnyEarsH3}
\end{center}
\end{figure}

\begin{prop}
Let $\pog$ be a periodic orbit framework where $G=(V,E)$ is a $P(2,1)$ graph, and $m$ is $\Torx$-constructive. Let $v_0$ be a three-valent 
vertex adjacent to two distinct vertices $v_1$ and $v_2$, and suppose that $v_0,v_1, v_2$ are circuit vertices. Let the edges adjacent to $v_0$ be 
\[\{v_0, v_1; m_0\}, \{v_0, v_2; m_1\}, \{v_0, v_2; m_2\}.\] Then 
$v_0$ is admissible unless $v_0, v_1, v_2$ are part of the bunny ears class of graphs, in which case we add the (loop) edge 
\[\{v_2, v_2; m_1 - m_2\}.\]
\label{prop:2neighbours}
\end{prop}

\begin{proof}
The candidate edges for a reverse $H2$ move deleting $v_0$ are
\[\{v_1, v_2; m_0+m_1\}, \{v_1, v_2; m_0 + m_2\}.\]
Toward a contradiction, suppose we cannot add either edge. Then there are subgraphs of $G$ which prevent the addition of these edges. 
Note first that there is no $(2,1)$-tight subgraph of $G$ containing $v_1$ and $v_2$ but {\it not} the vertex $v_0$, since this would 
contradict the fact that $G$ is $(2,1)$-tight. 

If we cannot add the edge $\{v_1, v_2; m_0 +m_1\}$, then we must have one of two scenarios: 
\begin{itemize}
	\item There is a subgraph $G_1^1 \subset G$ which contains $v_1, v_2$ but not $v_0$, and satisfying $i(V_1^1) = 2|V_1^1|-2$ that 
will not have an $x$-constructive gain assignment with the addition of the candidate edge $\{v_1, v_2; m_0 +m_1\}$. That is, $G_1^1$ is 
$(2,2)$-tight, contains no $x$-constructive cycles, and all paths from $v_1$ to $v_2$ have $x$-gain $(m_0+m_1)_x$. Note that $G_1^1$ 
cannot contain a $(2,1)$ subgraph, since that subgraph would be $x$-constructive by hypothesis.
	\item There is a $(2,3)$-tight subgraph $G_1^2 \subset G$ which contains $v_1, v_2$ but not $v_0$, that will not have a 
constructive gain assignment with the addition of $\{v_1, v_2; m_0 +m_1\}$. That is, $G_1^2$ contains no constructive cycles, and all 
paths from $v_1$ to $v_2$ have net gain $m_0+m_1$. 
\end{itemize}
Similarly, if we cannot add the edge $\{v_1, v_2; m_0 +m_2\}$, then we have the two scenarios:
\begin{itemize}
	\item $G_2^1$ is a $(2,2)$-tight subgraph containing $v_1, v_2$ but not $v_0$, and with all paths from $v_1$ to $v_2$ having 
$x$-gain $(m_0 + m_2)_x$. 
	\item $G_2^2$ is a $(2,3)$-tight subgraph containing $v_1, v_2$ but not $v_0$, and with all paths from $v_1$ to $v_2$ having 
net gain $m_0 + m_2$. 
\end{itemize}

The subscripts of the subgraphs correspond to which edge we are trying to put in, and the superscripts tell us how many edges we can 
add before creating an overbraced framework. If we cannot put in either edge, then we have three cases, corresponding to the possible 
pairs of the four subgraphs above:
\begin{enumerate}[(1)]
	\item $G_1^1$ and $G_2^1$
	\item $G_1^1$ and $G_2^2$ (and by symmetry $G_1^2$ and $G_2^1$)
	\item $G_1^2$ and $G_2^2$
\end{enumerate}

\noindent {\bf Case (1)}
$V_1^1 \cap V_2^1$ and $V_1^1 \cup V_2^1$ are both critical, otherwise adding back the vertex $v_0$ and its three adjacent edges provides a contradiction. Further $G_1^1 \cap G_2^1$ is connected by Lemma \ref{connectedlem} (2).
Therefore, any path from $v_1$ to $v_2$ in $G_1^1 \cap G_2^1$ is also a path in both $G_1^1$ and $G_2^1$. Hence all paths from $v_1$ to $v_2$ have $x$-gain $(m_0+m_1)_x = (m_0+m_2)_x$, and therefore $(m_1)_x=(m_2)_x$. 

Now let $G^*$ be the graph formed from $G_1^1 \cap G_2^1$ together with $v_0$ and the three adjacent edges. $G^*$ satisfies 
$i(V^*) = 2|V^*| - 1$, and it is not $x$-constructive, a contradiction. \\

\noindent {\bf Case (2)}
Since $G_1^1$ is $(2,2)$-tight and $G_2^2$ is $(2,3)$ tight, the intersection $G_1^1\cap G_2^2$ satisfies 
\[i(V_1^1 \cap V_2^2) = 2|V_1^1 \cap V_2^2| - \ell, \textrm{\ where } \ell \in \{2,3\}.\]
If $\ell = 2$, then $V_1^1 \cap V_2^2$ is critical but $G_2^2$ is $(2,3)$-tight, a contradiction. 

Therefore $\ell = 3$, and $G_1^1 \cap G_2^2$ is connected by Lemma \ref{connectedlem} (3). All paths from $v_1$ to $v_2$ have $x$-gain $(m_0+m_1)_x$, and net gain $m_0+m_2$. As in the previous case, $(m_1)_x = (m_2)_x$. 

Now consider the union $V_1^1 \cup V_2^2$, which is critical. Let $G^*$ be the $(2,1)$-tight graph created by adding the vertex $v_0$ and its three adjacent edges to $G_1^1 \cup G_2^2$. As a subgraph of $G$, the gain assignment on $G^*$ must be $x$-constructive. 

But all paths from $v_1$ to $v_2$ have $x$-gain $(m_0+m_1)_x = (m_0+m_2)_x$, and the gain assignment on $G^*$ is not $x$-constructive, a contradiction.
 
\noindent {\bf Case (3)}
$G_1^2 \cap G_2^2$ satisfies
\[i(V_1^1 \cap V_2^2) = 2|V_1^1 \cap V_2^2| - \ell, \textrm{\ where } \ell \in \{2,3, 4\}.\]
By the same argument as before $\ell \neq 2$.

Suppose $\ell = 3$. Lemma \ref{connectedlem} (3) implies the intersection $G_1^2\cap G_2^2$ is connected, and all paths in $G_1^2 \cap G_2^2$ have net gain $m_0 + m_1 = m_0 + m_2$, which implies that $m_1 = m_2$. Letting $G^*$ be the graph created from $G_1^2 \cap G_2^2$ by adding the vertex $v_0$ and its three adjacent edges, we find that $G^*$ is a $(2,2)$-tight subgraph of $G$ which is not constructive, a contradiction. 

Thus $\ell = 4$. When $|V_1^2 \cap V_2^2| > 2$, the intersection is connected, and we again obtain $m_1 = m_2$. As in Case (2), considering the graph $G_1^2 \cup G_2^2$ provides a contradiction. 

Finally we consider the case when $|V_1^2 \cap V_2^2| = 2$, in which case the intersection $E_1^2 \cap E_2^2$ is empty. This corresponds to the bunny ears class of graphs. Then the graph $G^*$ formed from adding $v_0$ and its three adjacent edges to $G_1^2 \cap G_2^2$ is $(2,1)$-tight, and moreover it must be the {\it minimal} $(2,1)$-tight subgraph (it does not contain any $(2,1)$-tight subgraphs). Therefore, there are no loops in $G$ (since loops always form the minimal $(2,1)$-tight subgraph), and we may add the loop edge $\{v_2, v_2; m_1-m_2\}$.
\end{proof}


Now observe that if at most one of $v_1,v_2$ lies in the circuit (and hence $v_0$ does not) then we may proceed exactly as for the fixed torus. This is clearly the case if $v_1,v_2$ are not in the circuit so suppose $v_1$ is. There are two cases, (1) there are two copies of $v_1v_0$ and (2) there are two copies of $v_2v_0$. In (1) the inverse Henneberg 2c operation creates a subgraph which is not $(2,1)$-sparse (the circuit with a loop added to $v_1$) and in (2) the inverse Henneberg 2c operation creates a second circuit contrary to uniqueness. Hence no possible inverse Henneberg operation alters the existing circuit.

\subsection{Three Distinct Neighbours}

\begin{prop}
Let $\pog$ be periodic orbit framework where $G$ is a $P(2,1)$ graph, and $m$ is $\Torx$-constructive. Let $v_0$ be a three-valent 
circuit vertex adjacent to three distinct vertices $v_1, v_2, v_3$, all of which are in the circuit. Let the three edges adjacent to $v_0$ be given by $\{v_0, v_i; m_i\}$. Then 
$v_0$ is admissible.
\label{prop:3distinct}
\end{prop}


\begin{proof}
We proceed in a similar fashion to the proof of Proposition \ref{prop:2neighbours}. 
The three candidate edges are
\[\{v_1, v_2; m_2-m_1\}, \{v_2, v_3; m_3 - m_2\}, \{v_3, v_1; m_1-m_3\}.\]
Suppose we cannot add any of these. We will consider four distinct cases. First we describe some notation.

If we cannot add the edge $\{v_1, v_2; m_2-m_1\}$, then we must have one of two scenarios: 
\begin{itemize}
	\item There is a subgraph $G_1^1 \subset G$ which contains $v_1, v_2$ but not $v_0, v_3$, such that $V_1^1$ is critical
that will not have an $x$-constructive gain assignment with the addition of the candidate edge $\{v_1, v_2; m_2-m_1\}$. That is, $G_1^1$ 
is $(2,2)$-tight, contains no $x$-constructive cycles, and all paths from $v_1$ to $v_2$ have $x$-gain $(m_2-m_1)_x$. Note that $G_1^1$ 
cannot contain a $(2,1)$ subgraph, since that subgraph would be $x$-constructive by hypothesis. Furthermore, $G_1^1$ must be constructive. 
	\item There is a $(2,3)$-tight subgraph $G_1^2 \subset G$ which contains $v_1, v_2$ but not $v_0, v_3$, that will not have a 
constructive gain assignment with the addition of $\{v_1, v_2; m_2-m_1\}$. That is, $G_1^2$ contains no constructive cycles, and all 
paths from $v_1$ to $v_2$ have net gain $m_0+m_1$. 
\end{itemize}
We remark that any $(2,1)$-tight subgraph containing $v_1$ and $v_2$ must be the whole circuit, so this situation does not arise. 
 
Similarly, if we cannot add the edge $\{v_2, v_3; m_3 - m_2\}$, then we have the two scenarios:
\begin{itemize}
	\item $G_2^1$ is a $(2,2)$-tight subgraph containing $v_2, v_3$ but not $v_0, v_1$, and with all paths from $v_2$ to $v_3$ having $x$-gain $(m_3 - m_2)_x$. 
	\item $G_2^2$ is a $(2,3)$-tight subgraph containing $v_2, v_3$ but not $v_0, v_1$, and with all paths from $v_2$ to $v_3$ having net gain $m_3 - m_2$. 
\end{itemize}
Finally if we cannot add the edge $\{v_3, v_1; m_1-m_3\}$, then we have the analogous subgraphs
\begin{itemize}
	\item $G_3^1$ is a $(2,2)$-tight subgraph containing $v_1, v_3$ but not $v_0, v_2$, and with all paths from $v_3$ to $v_1$ having $x$-gain $(m_1 - m_3)_x$. 
	\item $G_3^2$ is a $(2,3)$-tight subgraph containing $v_1, v_3$ but not $v_0, v_2$, and with all paths from $v_3$ to $v_1$ having net gain $m_1 - m_3$. 
\end{itemize}

We consider triples of subgraphs which prevent the addition of any edge, and we have four cases.
\begin{enumerate}[(1)]
	\item $G_1^1, G_2^1, G_3^1$
	\item $G_1^1, G_2^1, G_3^2$
	\item $G_1^1, G_2^2, G_3^2$
	\item $G_1^2, G_2^2, G_3^2$
\end{enumerate}
The other combinations follow by symmetry. 
In fact we will show that we can treat cases (1) and (2) together. \\

\noindent {\bf Cases (1), (2)}\\
Consider the intersection of $G_1^1$ and $G_2^1$. We show that these two subgraphs cannot co-exist, which eliminates these first two cases. 

First note that $i(V_1^1 \cap V_2^1) \leq 2|V_1^1 \cap V_2^1| - 2$ as a subgraph of the $(2,2)$-tight graphs $G_1^1, G_2^1$. In addition $|E_1^1 \cup E_2^1| \leq 2|V_1^1 \cup V_2^1| - 2$,  since $v_1, v_2, v_3 \in V_1^1 \cup V_2^1$, and the addition of $v_0$ cannot create an overbraced subgraph. Together these facts mean that we have equality in both cases. In addition, the intersection graph is connected. 

Now consider the graph $G_1^1 \cup G_2^1$. 
Lemma \ref{pathslem1} implies all paths from $v_1$ to $v_3$ have $x$-gain $(m_3 - m_1)_x$.

We also claim that $G_1^1 \cup G_2^2$ contains no $x$-constructive cycles. This can be proved using a similar argument. 

Continuing with Cases (1) and (2), let $G^*$ be the graph formed by appending the vertex $v_0$ and its three adjacent edges to $G_1^1 \cap G_2^1$. Since $G_1^1 \cap G_2^1$ is $(2,2)$-tight, $G^*$ is a $(2, 1)$-tight subgraph of $G$. As such, it must be $x$-constructive. But $G_1^1 \cap G_2^1$ does not contain any $x$-constructive cycles, and adding $v_0$ does not create any new constructive cycles, by our claim. This is a contradiction, and therefore $G_1^1$ and $G_2^1$ cannot both exist. \\

\noindent {\bf Case (3)}
Consider the intersection $G_1^1 \cap G_2^2$. Lemma \ref{criticallem1} (2) gives two cases.
{\it Case A.} 
$G_1^1 \cup G_2^2$ is $(2,2)$-tight, and 
$G_1^1 \cap G_2^2$ is $(2,3)$-tight (and hence connected).  

We may now argue along the same lines as Cases (1) and (2). That is, $G_1^1 \cup G_2^2$ is not $x$-constructive, and nor is the graph $G^*$ formed by appending $v_0$ and its three adjacent edges to $G_1^1 \cup G_2^2$, which is the contradiction, since $G^*$ is a $(2, 1)$-tight subgraph of $G$. \\

\noindent {\it Case B.} 
$ G_1^1 \cup G_2^2$ is semi-critical.
Let $G_4 = G_1^1 \cup G_2^2$ and consider $G_4 \cap G_3^2$. Since $V_4$ is semi-critical, and $G_3^2$ is $(2,3)$-tight, we know that 
\begin{equation} i(V_4 \cap V_3^2) + i(V_4 \cup V_3^2) = 2|V_4 \cup V_3^2| + 2|V_4 \cap V_3^2| - 6.\label{eq:case5} \end{equation}
Since $v_1, v_2, v_3 \in V_4 \cup V_3^2$, but $v_0$ is not, it must be the case that
\[2|V_4 \cup V_3^2|-4 \leq i(V_4 \cup V_3^2) \leq 2|V_4 \cup V_3^2| - 2. \]
Hence there are $3$ cases to analyse. \\

\noindent {\it Case I)} If we have equality,  $i(V_4 \cup V_3^2) = 2|V_4 \cup V_3^2| - 2$, then 
\begin{equation} i(V_4 \cap V_3^2) = 2|V_4 \cap V_3^2| - 4
\label{eq:case5B}
\end{equation} by (\ref{eq:case5}). In this case, the induced subgraph on the vertices $V_4 \cup V_3^2$ contains no more than 
$i(V_4 \cup V_3^2)$ edges, since otherwise we would have an overbraced subgraph. 

\begin{figure}[htbp]
\begin{center}
\includegraphics[width=2.5in]{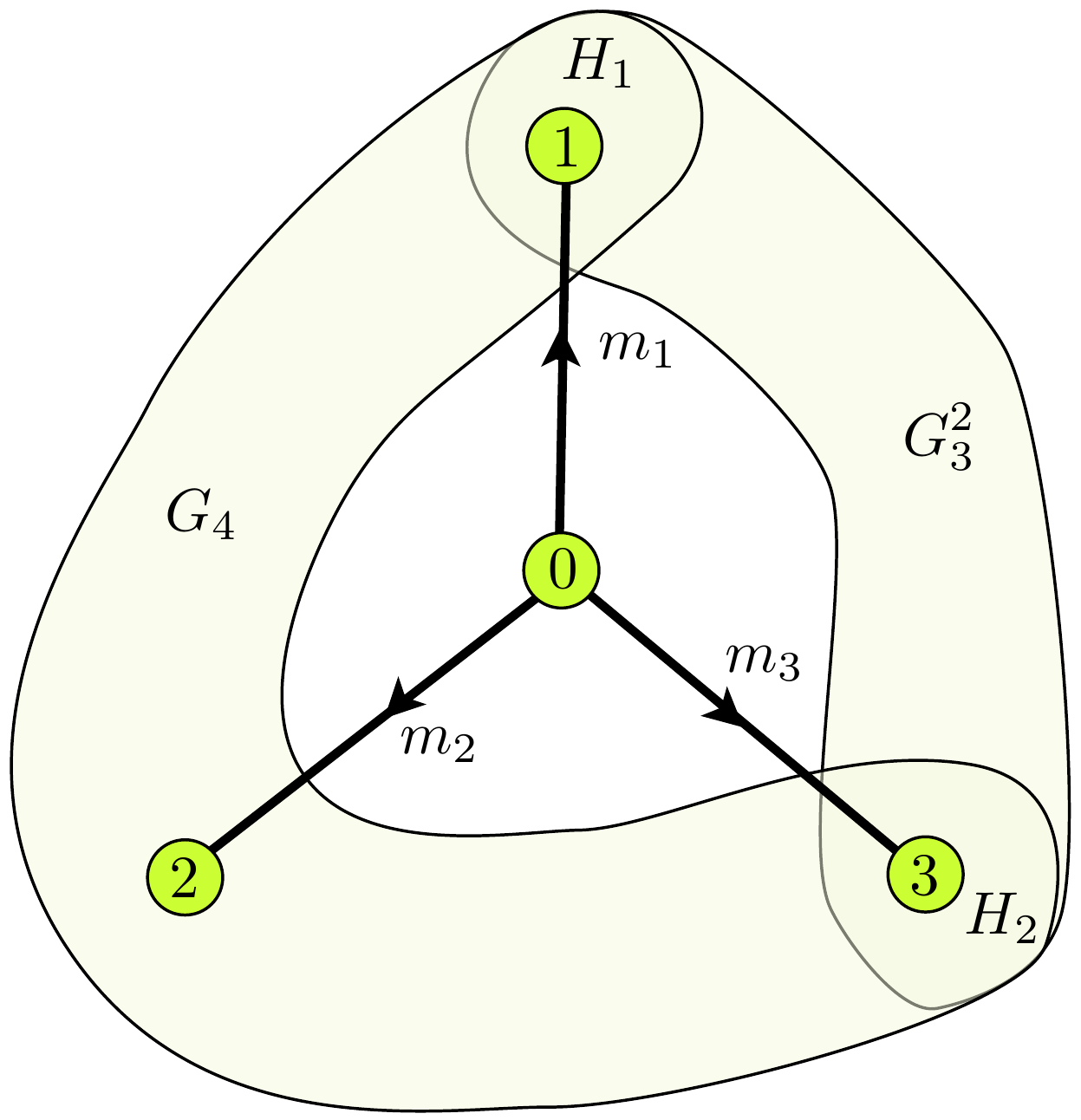}
\caption{Proof of Proposition \ref{prop:3distinct}, Case 3,B,I. $G_4 = G_1^1 \cup G_2^2$, and the intersection $G_4 \cap G_3^2$ is disconnected into at least two components, including $H_1$ and $H_2$. In Case 3,B,II, the intersection is connected, in which case it consists of a single edge linking $H_1$ and $H_2$. }
\label{fig:G4disconnected}
\end{center}
\end{figure}

It follows that the intersection, $G_4 \cap G_3^2$ is disconnected, since $v_2 \notin G_4 \cap G_3^2$(Figure \ref{fig:G4disconnected}). Say the component 
that lies completely within $G_1^1$ is called $H_1$, and the component lying completely within $G_2^2$ is called $H_2$. If $V(H_1)$ and 
$V(H_2)$ both contain more than one vertex, then we immediately obtain a contradiction, with the following argument. Since 
$|E(H_1)| \leq 2|V(H_1)| - 2$, and $|E(H_2)| \leq 2|V(H_2)|-3$, then 
\[i(V_4 \cap V_3^2) = |E(H_1)| + |E(H_2)| \leq 2(|V(H_1)| + |V(H_2)|) - 5.\]
But this contradicts (\ref{eq:case5B}). 

Therefore, it must be the case that one or both of $H_1, H_2$ consists only of a single vertex. Suppose $|V(H_2)| = 1$, and therefore 
$|E(H_2)| = 0$. Then by (\ref{eq:case5B}), we find that $|E(H_1)| = 2|V(H_1)| - 2$ (and may therefore be a single vertex as well, but 
need not be). Note all paths from $v_1$ to $v_3$ must pass through the vertex $v_2$.
Thus $G_4 \cup G_3^2$ is a $(2,2)$-tight subgraph such that all paths from $v_i$ to $v_j$ have gain $m_j - m_i$ for $1\leq i <j\leq 3$.
Let $G^*$ be the graph formed from $G_4 \cap G_3^2$ by adding the vertex $v_0$ and its three adjacent edges. Then $G^*$ has 
$i(V^*) = 2|V^*| -1$, but no $x$-constructive cycles. This is a contradiction, since we assumed that $G$ has no such subgraphs. \\

{\it Case II)} If it is the case that $i(V_4 \cup V_3^2) = 2|V_4 \cup V_3^2| - 3$, then 
\begin{equation} i(V_4 \cap V_3^2) = 2|V_4 \cap V_3^2| - 3
\label{eq:case5BII}
\end{equation} by (\ref{eq:case5}).
The intersection, $G_4 \cap G_3^2$ must be connected, since otherwise we would have overbraced subgraphs of $G_1^1$ or $G_2^2$. In particular, there is exactly one edge, say $e$, which connects $H_1 \subset G_1^1 \cap G_3^2$ with $H_2 \subset G_2^2 \cap G_3^2$ (see Figure \ref{fig:G4disconnected}). Moreover, since $|E(H_1)\leq 2|V(H_1)|-2$, it follows that $H_2$ consists of a single vertex, namely $v_3$. $H_1$ is then $(2,2)$-tight. 

We claim that the graph $G_4 \cup \{e\}$ is a $(2,2)$-tight subgraph of $G$ which contains no $x$-constructive cycles, and 
all paths from $v_i$ to $v_j$ have gain $m_j - m_i$ for $1\leq i <j\leq 3$.
These facts follow from similar arguments to those used in Cases 1 and 2. For example, any path from $v_1$ to $v_2$ that does not lie completely within $G_1^1$ must pass through the new edge $e$. But then $e$ joins $H_1 \subset G_1^1$ at some vertex $u$ which is also contained within $G_3^2$. Since $H_1$ is connected, there is a path from $u$ to $v_1$ within $H_1$. The path $v_3 \rightarrow u \rightarrow v_1$ lies within $G_3^2$, and therefore the net gain on this path is $m_1 - m_3$. The path from $v_2$ to $v_3$ is therefore: 
\[v_2 \rightarrow u \rightarrow v_3\]
\[\underbrace{v_2 \rightarrow u \rightarrow}_{(m_1 - m_2)_x} v_1 \underbrace{\rightarrow u \rightarrow v_3}_{(m_3-m_1)_x}, \]
and hence the net $x$-gain from $v_2$ to $v_3$ on any path that goes through the edge $e$ is $(m_3 - m_2)_x$, which proves (a). Similar arguments apply to show (b) and (c). \\

\noindent {\it Case III)} If it is the case that $i(V_4 \cup V_3^2) = 2|V_4 \cup V_3^2| - 4$, then 
\begin{equation} i(V_4 \cap V_3^2) = 2|V_4 \cap V_3^2| - 2
\label{eq:case5BIII}
\end{equation} by (\ref{eq:case5}).
As before we find that $H_1 \subset G_1^1 \cap G_3^2$ is $(2,2)$-tight, and $H_2 \subset G_2^2 \cap G_3^2$ consists of a single vertex. In addition, $G_4 \cap G_3^2$ contains two additional edges, say $e$ and $f$, as a consequence of (\ref{eq:case5BII}). But then $G_4 \cup \{e, f\}$ is a $(2,1)$-tight subgraph containing $v_1, v_2, v_3$ but not $v_0$, a contradiction \\

\noindent {\bf Case (4)} 

Consider the intersection of $G_1^2$ and $G_2^2$. From the previous arguments, we have the following cases:
\begin{enumerate}[{\it A.}]
\item $|V_1^2 \cap V_2^2| > 1$, and $i(V_1^2 \cap V_2^2) \leq 2|V_1^2 \cap V_2^2| - 3$, or
\item $|V_1^2 \cap V_2^2| = 1$, and $i(V_2^1 \cap V_2^2) = 0$. 
\end{enumerate}

{\it Case A.}
We know that $G_1^2 \cap G_2^2$ has 
\[i(V_1^2 \cap E_2^2) \leq 2|V_1^2 \cap V_2^2| - 3, \textrm{ and }\]
\[i(V_1^2 \cup E_2^2) \geq 2|V_1^2 \cup V_2^2| - 3.\]
However, since the union does not contain any constructive cycles (as the union of two graphs without constructive cycles by the arguments of Cases 1 and 2), in fact these inequalities are both tight. 

Let $G^*$ be the graph formed from $G_1^2 \cup G_2^2$ together with $v_0$ and its three adjacent edges.  Then $|E^*| = 2|V^*| - 2$, and must therefore be constructive as a subgraph of $G$. But, by the arguments of Cases (1) -- (2), we find that all paths from $v_1$ to $v_3$ have net gain $m_3 - m_1$, and therefore, $G^*$ contains no constructive cycles, a contradiction. 

{\it Case B.} 
Now suppose that $|V_1^2 \cap V_2^2| = 1$ (i.e. $V_1^2 \cap V_2^2 = \{v_2\}$), and therefore 
$i(V_1^2 \cup V_2^2) = 2|V_1^2 \cup V_1^2| - 4$. 


Let $G_4 = G_1^2 \cup G_2^2$ and consider $G_4 \cap G_3^2$. Since $G_4$ has $|E_4| = 2|V_4| - 4$, and $G_3^2$ is $(2,3)$-tight, we know that 
\begin{equation} i(V_4 \cap V_3^2) + i(V_4 \cup V_3^2) = 2|V_4 \cup V_3^2| + 2|V_4 \cap V_3^2| - 7.\label{eq:case7} \end{equation}
Since $v_1, v_2, v_3 \in V_4 \cup V_3^2$, but $v_0$ is not, it must be the case that
\[i(V_4 \cup V_3^2) \leq 2|V_4 \cup V_3^2| - 2. \]

Here we again have three cases, depending on the number of edges in $G_4 \cup G_3^2$. \\

\noindent {\it Case I)} If $i(V_4 \cup V_3^2) = 2|V_4 \cup V_3^2| - 2$, then $i(V_4 \cap V_3^2) \leq 2|V_4 \cap V_3^2| - 5$. The intersection is therefore disconnected, and must consist of a single vertex $H_1 \subset G_1^2$, where $H_1 = \{v_1\}$, and a $(2,3)$-tight subgraph $H_2 \subset G_3^2$ (i.e. $H_2$ is not a singleton). But then $G_4 \cup G_3^2$ is a $(2,2)$-tight subgraph of $G$ with no constructive cycles, a contradiction. \\

\noindent {\it Case II)} If $i(V_4 \cup V_3^2) = 2|V_4 \cup V_3^2| - 3$, then $i(V_4 \cap V_3^2) \leq 2|V_4 \cap V_3^2| - 4$. The intersection may be connected or disconnected, with $H_1 \subset G_1^2$, and $H_2 \subset G_3^2$ being either single vertices, or $(2,3)$-tight subgraphs. In any case, let $G^*$ be the subgraph of $G$ consisting of $G_4 \cup G_3^2$ together with $v_0$ and its three adjacent edges. Then $G^*$ is a $(2,2)$-tight subgraph of $G$ with no constructive cycles, a contradiction. \\

\noindent {\it Case III)} If $i(V_4 \cup V_3^2) = 2|V_4 \cup V_3^2| - 4$, then $i(V_4 \cap V_3^2) \leq 2|V_4 \cap V_3^2| - 3$. The intersection must be connected, and therefore $E(V_4 \cap V_3^2)$ contains more edges than $E_4 \cap E_3^2$. There are at most two induced edges, otherwise $G_4$ together with the three induced edges would form a $(2,1)$-tight subgraph containing $v_1, v_2, v_3$, a contradiction. 

If there are exactly two induced edges, then $G_4$ together with the induced edges $\{e, f\}$ is a $(2,2)$-tight subgraph of $G$. By arguments similar to those in Case 3, this is a subgraph with no constructive cycles, a contradiction. 

If there is exactly one induced edge, then $G_4$ together with the induced edge $\{e\}$ is a $(2,3)$-tight subgraph with no constructive 
cycles, and all paths from $v_i$ to $v_j$ have gain $m_j - m_i$ for $1\leq i <j\leq 3$.
Letting $G^*$ be the graph formed from $G_4 \cup \{e\}$ by adding $v_0$ and its three adjacent edges. Then $G^*$ is a $(2,2)$-tight subgraph of $G$ with no constructive cycles, a contradiction. 

\end{proof}


When two of the vertices lie in the circuit:

\begin{prop}
Let $\pog$ be periodic orbit framework where $G$ is a $P(2,1)$ graph, and $m$ is $\Torx$-constructive. Let $v_0$ be a three-valent 
vertex adjacent to three distinct vertices $v_1, v_2, v_3$, where the edges adjacent to $v_0$ are $\{v_0, v_i; m_i\}$. Suppose that $v_1, v_2$ are circuit vertices, but $v_0$ and $v_3$ lie outside the 
circuit.  Then $v_0$ is admissible.
\label{prop:3distinct2circuit}
\end{prop}

\begin{proof}
The three candidate edges are
\[\{v_1, v_2; m_2-m_1\}, \{v_2, v_3; m_3 - m_2\}, \{v_3, v_1; m_1-m_3\}.\]
Suppose $G_1^0$ is a $(2,1)$-tight subgraph of $G$ containing $v_1, v_2$, but not $v_0$ or $v_3$. 
If $G_2^1$ or $G_3^1$ exist we contradict Lemma \ref{criticallem2} (2).
Similarly if $G_2^2$ exists and intersects the circuit in more than one vertex we contradict Lemma \ref{criticallem2} (3).
Finally if $|V_1^0\cap V_2^2|=1$ then we contradict Lemma \ref{semicriticallemma}.

In other words, if two of the vertices $v_1, v_2, v_3$ lie in the circuit (or any $(2,1)$-tight subgraph containing the circuit), then 
at least one of the other two candidate edges may always be added. 
\end{proof}

\begin{prop}
Let $\pog$ be periodic orbit framework where $G=(V,E)$ is a $P(2,1)$ graph with unique over-critical set $X$, and $m$ is 
$\Torx$-constructive. Let $v_0\in V$ have
$N(v_0)=\{v_1,v_2,v_3\}$ and suppose $|X\cap N(v_0)|\leq 1$. Then $v_0$ is admissible.
\label{prop:3distinct1circuit}
\end{prop}

\begin{proof}
Lemma \ref{criticallem3} reduces this to a fixed torus problem, thus the result follows from Theorem \ref{thm:fixedtorus}.
\end{proof}

We now reach the stated goal of this section.

\begin{proof}[Proof of Theorem \ref{thm:mainresult}]
One direction follows from the definitions of the gain-preserving Henneberg operations. The converse follows from Propositions \ref{prop:2neighbours}, \ref{prop:3distinct}, \ref{prop:3distinct2circuit} and \ref{prop:3distinct1circuit}.
\end{proof}

\section{Extensions}
\label{sec:extensions}
Up to this point, we have focused on frameworks which are on a variable torus, where the variability is in the $x$-direction only. It is immediate that we can apply this result to frameworks which are variable in the $y$-direction only, simply by demanding that our framework be $\mathcal T_y$-constructive. We now outline several other variations on the torus with one degree of freedom. 

\subsection{Fixed Area and Angle}
One variation of the variable torus is the torus whose area remains fixed, as does the angle between the two generators, which we denote $\T_{vol}^2$. It is generated by the lattice matrix 
\[L_{vol}(t) =  \left(\begin{array}{cc}x(t) & 0 \\0 & kx(t)\end{array}\right).\] 
In this case the angle is constrained to remain fixed at $\pi/2$. 

Certainly, if a framework is generically rigid on $\Torx^2$ or $\T_y^2$, then it is generically rigid on $\T_{vol}^2$ too. Of course, the natural way to understand the fixed area setting is to consider frameworks on the torus with two degrees of freedom, which we do not address here. 

\subsection{Flexible Angle}
Let $\T_{\theta}^2$ be the torus generated by the lattice matrix
\[L_{\theta} = \left(\begin{array}{cc}1 & 0 \\\cos \theta & \sin \theta \end{array}\right).\]
This is the torus that has generators with fixed lengths, and a variable angle between them. Then the rigidity matrix for frameworks on $\T_{\theta}^2$ has a column corresponding to the variable $\theta(t)$. Instead of requiring an $x$-constructive cycle, the necessary condition for rigidity can be seen to be that the critical subgraph contains a cycle with net gain $(m_1, m_2)$ satisfying $m_1m_2 \neq 0$. The other requirements for rigidity are as for frameworks on $\T_x^2$. 

\subsection{Frameworks on a Variable Cylinder}

As a direct consequence of Theorem \ref{thm:rigidityinduction} we obtain a characterisation of the generic rigidity of frameworks which are periodic in one direction only. That is, Theorem \ref{thm:rigidityinduction} provides necessary and sufficient conditions for the rigidity of frameworks on the cylinder with variable circumference, where the cylinder is a ``flat cylinder." That is, we are {\it not} considering frameworks supported on surfaces in three dimensional space as in \cite{NixonOwenPower}. See Figure \ref{fig:flexCylinder} for an example. Such frameworks are similar to {\it frieze patterns}, although we assume that such patterns exhibit {\it only} translational symmetry (in one direction), and do not have any of the other symmetries of frieze patterns. 

\begin{figure}
\begin{center}
\subfloat[]{\begin{tikzpicture}[scale=2] 
\tikzstyle{vertex1}=[circle, draw, fill=couch, inner sep=1pt, minimum width=4pt]; 
\tikzstyle{vertex2}=[circle, draw, fill=lips, inner sep=1pt, minimum width=4pt]; 

\clip (-2.2, -.6) -- (2.2, -.6) -- (2.2, .6) -- (-2.2, .6) -- (-2.2, -.6);

\foreach \x in {-2, -1, 0,1,2} 
\foreach \y in {  0} 
{ 
\node[vertex1] (1\x\y) at (\x-.25, \y-.25){};
} 

\foreach \x in {-2, -1, 0,1,2} 
\foreach \y in {0} 
{ 
\node[vertex1] (2\x\y) at (\x,\y) {}; 
} 

\foreach \x in {-2, -1, 0,1,2} 
\foreach \y in {0} 
{ 
\node[vertex1] (3\x\y) at (\x-.1,\y-.35) {}; 
} 

\foreach \x in {-2, -1, 0,1,2} 
\foreach \y in {0} 
{ 
\node[vertex1] (4\x\y) at (\x+.35,\y-.15) {}; 
} 

\foreach \x in {-2, -1, 0,1,2} 
\foreach \y in {0} 
{ 
\node[vertex1] (5\x\y) at (\x+.05,\y+.3) {}; 
}

\draw \foreach \x in { -2,-1, 0,1, 2} 
\foreach \y in { 0} {(1\x\y) -- (2\x\y)--(3\x\y)--(1\x\y) (3\x\y) -- (4\x\y) -- (5\x\y) -- (2\x\y)};
%
%
%
\draw \foreach \x in { -2,-1, 0,1,2} 
\foreach \y in { 0} {(4\x\y) -- (\x+1, \y) };

\draw \foreach \x in { -2,-1, 0,1,2} 
\foreach \y in { 0} {(5\x\y) -- (\x-1, \y) };

\draw (400) -- (110) (4-20) -- (1-10);%
\draw (4-10) -- (100) (410) -- (120);

Redrawing vertex 1 for prettiness
\foreach \x in {-2, -1, 0,1,2,3} 
\foreach \y in { 0} 
{ 
\node[vertex1] (2\x\y) at (\x, \y) {};
} 
	\draw[sky, very thick, dashed] (-.5, -.5) -- (.5, -.5);
	\draw[sky, very thick, dashed] (-.5, .5) -- (.5, .5);
	\draw[sky, very thick] (-.5, -.5) -- (-.5, .5);
	\draw[sky, very thick] (.5, -.5) -- (.5, .5);

\end{tikzpicture}}\hspace{1cm}
\subfloat[]{\begin{tikzpicture}[scale=4, thick, font=\footnotesize] 
\tikzstyle{vertex1}=[circle, draw, fill=couch, inner sep=1pt, minimum width=3pt]; 
\tikzstyle{vertex2}=[circle, draw, fill=lips, inner sep=1pt, minimum width=3pt]; 
\tikzstyle{gain} = [fill=white, inner sep = 0pt,  font=\scriptsize, anchor=center];

\node[vertex1] (1) at (-.5, -.25){1};
 \node[vertex1] (2) at (-.1,.1) {2}; 
\node[vertex1] (3) at (-.2,-.35) {3}; 
\node[vertex1] (4) at (.35,-.15) {4}; 
\node[vertex1] (5) at (.2,.4) {5}; 

\draw[thick] (1) -- (2) -- (3) -- (1) (2) edge [bend right] (5);
\draw (5) -- (4) -- (3);

\pgfsetarrowsend{stealth}
\draw (2) edge [bend left] node[gain] {$1$} (5) 
(4) edge node[gain] {$1$} (2)
(4) edge node[gain] {$1$} (1);

\end{tikzpicture}}

\caption{A framework which is periodic in one direction only (a), and its gain graph (b) which is labelled by elements of $\mathbb Z$. \label{fig:flexCylinder}}
\end{center}
\end{figure}
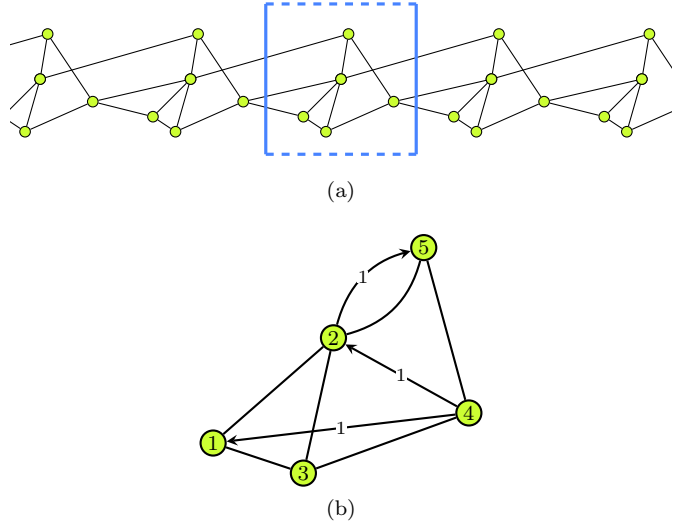

Let $\pog$ be a gain graph with $m: E \rightarrow \mathbb Z$. For each edge $e = \{i,j;m_e\}$, the integer $m_e$ now represents the number of times the edge $e$ ``wraps" around the cylinder.  Let $p:V \rightarrow \mathbb R^2 / (\mathbb Z \times Id)$. That is, for $v \in V$, $p(v) \in [0,1) \times \mathbb R$. Let $\C =  [0,1) \times \mathbb R$, and we call this the {\it variable cylinder}. 

Let $\pog$ be a gain graph with gain assignments from $\mathbb Z$. Let $\hat{m}: E \rightarrow \mathbb Z^2$ be the gain assignment on $G$ given by $\hat{m}(e) = (m(e), 0)$. Since the rigidity matrix for the cylinder and the rigidity matrix for $\mathcal T_x^2$ are identical (each has exactly one lattice column), we have the following proposition. 

\begin{prop}
A periodic orbit graph $\pog$ is generically rigid on the variable cylinder $\C$ if and only if $\langle G, \hat{m} \rangle$ is generically rigid on $\T_x^2$. 
\end{prop}

\section{Further Work}

\subsection{Fully Variable Torus}

Generic minimal rigidity on the fully variable torus $\T^2$ has been characterised by Malestein and Theran \cite{MalesteinTheran}; though their proof is non-inductive. There is a significant new challenge to providing such a constructive characterisation as the underlying graphs can have minimum degree $4$. This suggests a consideration of degree $4$ Henneberg type operations such as $X$ and $V$-replacement, however these operations are already known to be problematic for $3$-dimensional rigidity, \cite{Whiteleyscene}. For periodic frameworks the main challenges are the large number of cases and the fact that the variants of $V$-replacement do not necessarily preserve the relevant counting conditions.

\subsection{Global Periodic Rigidity}

The global rigidity of frameworks in Euclidean space is well studied. As in the case of rigidity there is a celebrated combinatorial characterisation in the plane, see \cite{JacksonJordan}, but no such characterisation is known in higher dimensions. The result in the plane relied on inductive constructions of circuits in the plane rigidity matroid \cite{BergJordan} and crucially of $3$-connected, redundantly rigid graphs \cite{JacksonJordan}. 
It is tempting then, as a first step, to take the inductive constructions provided here and in \cite{ThesisPaper2} and try to provide related constructions for the circuits in the fixed (or flexible) torus rigidity matroid. 
As far as we know this has not yet been addressed.

\subsection{Periodic Body-Bar Frameworks}
Recently, a characterisation of the generic rigidity of periodic body-bar frameworks on the fixed torus has been developed, when $d \leq 3$ \cite{RossBodyBar}. The body-bar setting is somewhat different from the present study, since we are no longer working within a combinatorial subclass of the full inductive class (as we are with $P(2,1)$ graphs in the class of all $(2,1)$-tight graphs). For this reason, it is possible to use existing inductive characterisations of the relevant combinatorial structures \cite{fekete}. However, it may be possible to extend those results to the partially variable torus, which may require a more subtle approach, as in the present paper. \\

\noindent {\it Acknowledgements.}

Much of this research was carried out while both authors were at the Fields Institute, University of Toronto. We enjoyed several stimulating conversations with Justin Malestein and Louis Theran on the broad topic of inductive constructions for periodic and symmetric frameworks. 

\bibliographystyle{abbrv} 
\bibliography{FlexTorusBib}

\end{document}